\documentclass[12pt,reqno,a4paper]{article}

\usepackage{graphicx}
\usepackage{xcolor}
\usepackage{cmap,mathtools}
\usepackage[T1]{fontenc}
\usepackage[utf8]{inputenc}
\usepackage[english]{babel}
\usepackage{amsfonts,amssymb,amsthm,amsmath,enumerate,bm}
\usepackage{url}
\usepackage[shortlabels]{enumitem}
\usepackage{float}
\usepackage{secdot}

\allowdisplaybreaks

\usepackage{a4wide}

\usepackage{xcolor}
\usepackage[colorlinks=true,linktocpage,pdfpagelabels,bookmarksnumbered,bookmarksopen]{hyperref}
\definecolor{ForestGreen}{rgb}{0.1,0.6,0.05}
\definecolor{EgyptBlue}{rgb}{0.063,0.1,0.6}
\hypersetup{
	colorlinks=true,
	linkcolor=EgyptBlue,         
	citecolor=ForestGreen,
	urlcolor=olive
}

\numberwithin{equation}{section}
\DeclarePairedDelimiter\abs{\lvert}{\rvert}%
\DeclarePairedDelimiter\norm{\lVert}{\rVert}%

\makeatletter
\let\oldabs\abs
\def\abs{\@ifstar{\oldabs}{\oldabs*}}

\let\oldnorm\norm
\def\norm{\@ifstar{\oldnorm}{\oldnorm*}}

\newcommand{\be} {\beta}

\newcommand{\Dep} {\Delta_p}

\newcommand{\Ga} {\Gamma}
\newcommand{\om} {\omega}
\newcommand{\Om} {\Omega}
\newcommand{\la} {\lambda}

\newcommand{\si} {\sigma}

\newcommand{\noi} {\noindent}

\newcommand{\var} {\varepsilon}
\newcommand{\ra} {\rightarrow}

\newcommand{\wra} {\rightharpoonup}
\newcommand{\wrastar} {\overset{\ast}{\rightharpoonup}}

\DeclareMathAlphabet{\mathpzc}{T1}{pzc}{m}{it}
\newtheorem{thm}{Theorem}
\newtheorem{definition}{Definition}
\newtheorem{proposition}{Proposition}
\newtheorem{corollary}{Corollary}
\newtheorem{rmk}{Remark}
\newtheorem{lem}{Lemma}

\def\w{{\widetilde w}}

\def\M{{\mathcal{M}(\Om)}}
\def\dr{{\rm d}r}
\def\cp{{\rm Cap}_p}

\def\Dp{{{\mathcal D}^{1,p}_0(\Om)}}
\def\w2r{{{ W}^{2,2}(\R^N)}}

\def\d2{{{\mathcal D}^{2,2}_0(\Om)}}
\def\cset{{\subset \subset }}

\def\A{{\mathcal A}}
\def\C{{\mathcal C}}
\def\D{{\mathcal D}}
\def\E{{\mathcal E}}
\def\H{{\mathcal{H}(\Om)}}

\def\R{{\mathbb R}}
\def\N{{\mathbb N}}
\def\F{{\mathcal F}}
\def\({{\Big(}}
\def\){{\Big)}}

\def\ws2{{\F_{\frac{N}{2}}}}
\def\L2{{ L^{1,\;\infty}(\log L)^2}}

\def\l2{\mathcal M\log L}

\def\c1Loc{{\C_{loc}^1}}
\title{The compactness and the concentration compactness via $p$-capacity}
\author{T. V. Anoop \,, Ujjal Das\thanks{Corresponding Author}}

\date{}

\begin{document}
\maketitle 

\begin{abstract}
	For  $p \in (1,N)$ and  $\Omega\subseteq \mathbb{R}^N$ open,  the Beppo-Levi space $\mathcal{D}^{1,p}_0(\Omega)$ is the completion of  $C_c^{\infty}(\Omega)$ with respect to the norm $\left[ \int_{\Omega}|\nabla u|^p \ dx \right]^ \frac{1}{p}.$  Using the $p$-capacity, we define a norm  and then  identify the Banach function space $\mathcal{H}(\Omega)$  with  the set of all  $g$ in $L^1_{loc}(\Omega)$ that admits the following Hardy-Sobolev type inequality: 
	\begin{eqnarray*}
		\int_{\Omega} |g| |u|^p \ dx \leq C  \int_{\Omega} |\nabla u|^p \ dx, \forall\; u \in \mathcal{D}^{1,p}_0(\Omega),
	\end{eqnarray*}
for some $C>0.$ Further, we characterize the set of all $g$ in  $\mathcal{H}(\Omega)$  for which the map $G(u)= \displaystyle \int_{\Omega} g |u|^p \ dx$ is compact on $\mathcal{D}^{1,p}_0(\Omega)$. We use a variation of the concentration compactness lemma to give  a sufficient condition on $g\in \mathcal{H}(\Omega)$  so that the  best constant in the above inequality is attained in $\mathcal{D}^{1,p}_0(\Omega)$. 
\end{abstract}

\section{Introduction}
 For  $p \in (1,N)$ and an open subset $\Om$  of $\R^N$, the Beppo-Levi space $\Dp$ is the completion of  $C_c^{\infty}(\Om)$ with respect to the norm, $ \norm{u}_{\D} :=\left[  \int_{\Omega}|\nabla u|^p \ dx \right]^ \frac{1}{p}.$ 
We look for the weight function  $g\in L^1_{loc}(\Om)$ that admits the following Hardy-Sobolev type inequality: 
 \begin{eqnarray}\label{HS}
  \int_{\Omega} g |u|^p \ dx \leq C  \int_{\Omega} |\nabla u|^p \ dx, \forall\; u \in \Dp,
 \end{eqnarray}
for some $C>0.$
\begin{definition}
 A function $g\in L^1_{loc}(\Om)$ is called a Hardy potentials if $|g|$ satisfies \eqref{HS}. We denote the space of Hardy potentials by $\H$. 
\end{definition} 
 
 Using Poincar\'{e} inequality, it is easy to verify that $L^{\infty}(\Om)\subseteq \H$ if $\Om$ is bounded (in one direction).  Further, the classical Hardy-Sobolev inequality
  \begin{equation} \label{CHS}
  \int_{\Om} \frac{1}{|x|^p} |u|^p \ dx \leq \displaystyle \left(\frac{p}{N-p} \right)^p \int_{\Om} |\nabla u|^p \ dx, \ u \in \Dp
  \end{equation}
  ensures that $\frac{1}{|x|^p}\in \H,$ even when $\Om$ contains the origin.  In the context of improving the Hardy-Sobolev inequality many examples of  Hardy potentials were produced,
  see \cite{MR1605678,Adimurthy_Mythily,Filippas} and the references there in. For $p=2$ and $\Om$ bounded,    $L^r(\Om)\subseteq \H$ with $r>\frac{N}{2}$  \cite{Manes-Micheletti}, $r=\frac{N}{2}$ \cite{Allegretto}. 
  For  $p\in (1,\infty)$ and for general domain $\Om,$ in \cite{Visciglia} authors showed that $L^{\frac{N}{p},\infty}(\Om)\subseteq \H$ using the Lorentz-Sobolev embedding. If $\Om$ is the exterior of closed unit ball, 
  then examples of Hardy potentials outside the $L^{\frac{N}{p},\infty}(\Om)$ are provided in \cite{ADS-exterior}. For $g\in L^1_{loc}(\Om)$, we consider \[\tilde{g}(r)= {\rm ess}\sup \{|g(y)|: |y|=r \}, \ r > 0,\]
  where the essential supremum is taken with respect to $(N-1)$ dimensional surface measure. 
  Let \[I(\Om)= \{ g\in L^1_{loc}(\Om) : \tilde{g} \in L^{1}((0, \infty),r^{p-1}\dr) \};\qquad   \norm{g}_I = \int_0^{\infty} r^{p-1} |\tilde{g}|(r) dr.\] 
  Then,  $I(\Om)$ is a Banach space with the norm $\norm{.}_I$ and $I(\Om)\subseteq \H$ (Proposition \ref{Iomega}).
  
   In \cite{Mazya}, Maz'ya gave a very intrinsic characterization of a Hardy potential using the $p$-capacity (see Section 2.4.1, page 128).
 Recall that, for $ F \cset \Omega,$ the $p$-capacity of $F$ relative to $\Om$ is defined as,
  \[{\cp(F,\Om)}  = \inf \left\{ \displaystyle \int_{\Omega} | \nabla u |^p \ dx : u \in  \mathcal{N} (F) \right \},\]
  where $ \mathcal{N} (F)= \{ u \in \Dp : u \geq 1 \ \mbox{in a neighbourhood of}\; F \}$. Thus for $g \in \H$ and $ w \in  \mathcal{N} (F)$, we have 
  \[\displaystyle \int_{F} |g| \ dx \leq  \int_{\Omega} |g||w|^p \ dx \leq C \int_{\Omega} |\nabla w|^p \ dx.\]
  Now by taking the infimum over $\mathcal{N}(F) $ and as $F$ is arbitrary, we get a necessary condition:
  \begin{eqnarray*}\label{norm}
   \displaystyle \sup_{ F \cset \Omega} \frac{\displaystyle \int_{F} |g| \ dx}{\cp(F,\Omega)}  \leq  C .
  \end{eqnarray*}
  Maz'ya proved that the above condition is also sufficient for $g$ to be in $\H.$ Motivated by this, for $g \in L^1_{loc}(\Om)$, we define, 
  \begin{eqnarray*}
  \norm{g}= \sup\left\{ \frac{\displaystyle \int_{F} |g| \ dx}{\cp(F,\Om)}:F \cset \Om; |F|\ne 0 \right\}.                           
   \end{eqnarray*}  
 One can verify that $\norm{.}$ is a Banach function norm on $\H $.  The Banach function space structure of $\H$ and Maz'ya's characterization helps  us to  prove an  embedding of $\Dp$ which is finer than the Lorentz-Sobolev 
 embedding  proved in \cite{Alvino}.  We also provide an alternate proof for the Lorentz-Sobolev embedding (Proposition \ref{embedding}). 
 
  For $g\in \H,$ let $B_g$ be the best constant  in \eqref{HS}.
In this article, we are  interested to find the Hardy potentials $g \in \H$ for which $B_g$ is attained in $\Dp$. 
Many authors have considered similar problems in the context of finding the first (least)  positive  eigenvalue for  the following weighted  eigenvalue problem:
  \begin{eqnarray} \label{EVP}
 -\Dep u & = & \la g |u|^{p-2}u \  \mbox{ on } \Dp. 
 \end{eqnarray}
 If the map  $G: \Dp \ra \R $ defined by  $G(u)= \displaystyle \int_{\Om} g|u|^p \ dx$ is compact, then a direct variational method ensures that the first positive eigenvalue for the above problem exists and 
 $B_g$  is attained in $\Dp$. For $p=2$ and $\Om$ bounded, the compactness of $G$ is proved for  $g\in L^r(\Om)$ with $r>\frac{N}{2}$ 
 in \cite{Manes-Micheletti} and $r=\frac{N}{2}$ in \cite{Allegretto}. For $p\in (1,\infty)$ and for general domain $\Om,$ $g\in L^{\frac{N}{p},d}(\Om)$ with $d<\infty$, in \cite{Visciglia}. 
 The result is extended for a larger space $\F_{\frac{N}{p}}(\Om):=\overline{C_c^{\infty}(\Om)}$ in $L^{\frac{N}{p},\infty}(\Om)$ in \cite{AMM} for $p=2$ and in \cite{anoop-p} for $p\in (1,N)$. 
In \cite{ADS-exterior}, authors obtained the compactness of $G$ for $g\in I(\overline{B}_1^c).$ 
 
 We extend and unify all the existing sufficient conditions for the compactness of $G$ by characterizing the set of all Hardy potentials 
 for which the map $G$  is compact on $\Dp$. In fact, we provide three different characterizations and each of them uses the Banach function space structure of $\H$ in one way or other.  
 Our first characterization is motivated by the definition of the space $\F_{\frac{N}{p}}(\Om)$ considered in \cite{AMM,anoop-p}. Here, we consider the following subspace of $\H$:
\[\F(\Om): =  \overline{C_c^{\infty}(\Om)} \text{ in }  \H.\]
  Now, we  have the following  theorem:
\begin{thm}\label{eqivthm}
 Let $g\in \H$. Then  $G:\Dp \ra \R$ is compact if and only if $g \in \F(\Om)$. 
\end{thm}

 For the second characterization, we use the notion of the absolute continuous norm on a Banach function space.
 
\begin{definition} \label{ABnorm}
  Let $X = (X(\Om),\norm{.}_X )$ be a Banach function space. A function $f\in X$ is said to have absolute continuous norm,  if for any sequence of measurable subsets $(A_n)$ of $\Om$
  with $ \chi_{A_n} $ converges to $ 0  $ a.e. on $\Om,$ then $\norm{f \chi_{A_n}}_X$ converges to $0$.
 \end{definition}
 \begin{thm} \label{eqivthm1}
   Let $g\in \H$. Then $G:\Dp \ra \R$ is compact if and only if $g$ has absolute continuous norm in $\H$.
 \end{thm}

 The third characterization is based on a concentration function that is defined using the norm on $\H$.  For $x \in \overline{\Om} $ and $r> 0$, 
 let $B_r(x)$ be the ball of radius $r$  centered at $x$. Now for $g\in \H$, we  define, 
 \begin{eqnarray*}
  \C_g(x)= \lim_{r \ra 0}  \norm{g \chi_{ B_r(x)}}, \qquad 
  \C_g({\infty}) =\lim_{R \ra \infty} \norm{g \chi_{ {B}_R(0)^c}}. 
 \end{eqnarray*}
 Observe that, the concentration function $\C_g$ measures the lack of absolute continuity of the norm of $g$ at all the points in  $\Om$ and at the infinity. Therefore, if   $\C_g$ vanishes everywhere, then  naturally one may anticipate 
 the compactness of $G,$ and precisely this is our next result.
  \begin{thm} \label{eqivthm2}
   Let $g\in \H$. Then $G:\Dp \ra \R$ is compact if and only if \[ \C_{g}(x)=0,\;\forall x\in \overline{\Om}\cup\{\infty\}.\]
 \end{thm}

Observe that the  best constant in \eqref{HS} is attained  in $\Dp$ if and only if the following minimization problem (dual problem) has a minimizer: 
  \begin{eqnarray}\label{Min1}
  \min \left\{ \int_{\Om} |\nabla u|^p\ dx : \ u\in \Dp, \ \int_{\Om} g|u|^p\ dx =1 \right\}.
 \end{eqnarray}
If $G$ is compact, then  the level set $G^{-1}\{1\}$ is weakly closed  and hence the weak limit of a minimizing sequence lie in $G^{-1}\{1\}.$ Indeed, the weak limit of a minimizing sequence solves the minimization problem 
and $B_g$ is attained at this weak limit. However,  for the existence of the weak limit of a minimizing sequence, it is not necessary to have $G^{-1}\{1\}$ is  weakly closed. In other words, for a 
 non-compact $G$, \eqref{Min1} may admit a minimizer. These cases were treated in  \cite{Tertikas} for $p=2$, $\Om = \R^N$ and in \cite{Smets} for  $p \in (1,N)$ and general $\Om$.
 In \cite{Tertikas}{Smets} and
 \cite{Smets}, authors provided sufficient condition on $g$ for the existence of minimizer for \eqref{Min1}. In \cite{Tertikas}, Tertikas used the celebrated concentration compactness lemma of Lions
(\cite{Lions1a,Lions2a}) and Smets proved a variant of this lemma in \cite{Smets}. One of their main restrictions was the countability of the closure of the `singular' set of $g$ (see Remark \ref{singularset} for their definition of a singular set). 
In this article, we define the singular set as
$\sum_g = \{x \in \overline{\Om}: \C_g(x)>0\}$ 
and in fact, $\sum_g$ coincides with the singular set considered by Tertikas \cite{Tertikas} and Smets \cite{Smets}. In the next theorem, we provide a sufficient condition which is weaker than the countability assumptions of \cite{Smets,Tertikas} for the existence of minimizer for \eqref{Min1}.
 \begin{thm}\label{exismin}
 Let $g\in \H$ be a non-negative function such that $\left|\overline{\sum_g} \right|=0$ and  \[C_H\C_{g}(x)< B_g, \forall x \in \overline{\Om} \cup \{ \infty \} ,\] where $\left|\overline{\sum_g} \right|$ denotes the Lebesgue measure of $\overline{\sum_g}$,
 $ B_g $ is the best constant in \eqref{HS} and $C_H=p^p(p-1)^{1-p}$. Then  $B_g$ is attained on $\Dp.$
\end{thm}

If  $g\in \H$ with $g\ge 0$, $\left|\overline{\sum_g} \right|=0$ and  $C_H  dist(g, \F(\Om)) < \norm{g}$, then by the above theorem, $B_g$ is attained on $\Dp$ (Corollary \ref{distrmk}). This helps us to produce Hardy potentials for which the map $G$ is not compact, 
however $B_g$ is attained.  The following theorem is an analogue of Theorem 1.3 of \cite{Tertikas}:
\begin{thm} \label{thmperturb}
 Let $h \in \H$ with $h \geq 0$ and $\left|\overline{\sum_h} \right|=0.$ Then for any non-zero, non-negative $ \phi \in \F(\Om)$, there exists $\epsilon_0 > 0$ such that $B_g$ is attained in $\Dp$ for  $g= h+ \epsilon \phi$, for all $ \epsilon > \epsilon_0.$
\end{thm}
\begin{rmk} \rm
$(i)$. We provide  cylindrical  Hardy potentials $g$ for which $|\sum_g|=0$, but $\sum_g$ is not countable (see Remark \ref{example}). Such cylindrical weights were considered by Badiale and Tarantello in \cite{Tarantello} (for $N=3$), Mancini et. al in \cite{Mancini} (for $N \geq 3$) to study certain semi-linear PDE involving Sobolev critical exponent. In astrophysics, such critical exponent problems with cylindrical weights often arises in the dynamics of galaxies \cite{Bertin,Ciotti}.

\noi $(ii)$. For a cylindrical Hardy potential $g \in \H$ with $|\overline{\sum_g}|=0$, one can consider its perturbation $\tilde{g}:=g+\phi$ by a suitable $\phi \in C_c^{\infty}(\Om)$ and apply the above theorem to ensure $B_{\tilde{g}}$ is attained in $\Dp$ (see Remark \ref{example} for a precise example). It is worth noticing that $|\overline{\sum_{\tilde{g}}}| = 0$ but not countable. Indeed,  the results of \cite{Tertikas,Smets} are not applicable for such Hardy potentials.
\end{rmk}

The rest of the paper is organized as follows. In Section~\ref{Prelim}, we recall some important results that are required for the development of this article. Further, 
we discuss the function spaces $\H$, $\F(\Om)$ and some embeddings of $\Dp$ in Section \ref{Em}.
 In Section~\ref{Compactness} we prove Theorem \ref{eqivthm}, Theorem \ref{eqivthm1} and Theorem \ref{eqivthm2}. Section~\ref{Existence} contains the proof of Theorem \ref{exismin} and Theorem \ref{thmperturb}.

\section{Preliminaries} \label{Prelim}
In this section, we  briefly outline the symmetrization, Banach function space, Lorentz spaces and $p$-capacity and list some of their properties. Further, we state a few other results that we  use in the subsequent sections. 
\subsection{Symmetrization}
Let $\Omega \subseteq \R^N$ be an open set.  Let $\M$ be the set of all extended real valued Lebesgue measurable functions that are finite a.e. in $\Om.$ For $f\in\M $ and for $s>0$, we define
$E_f(s)=\{x: |f(x)|>s \}$ and the distribution function $\alpha_f$ of $f$ is defined as 
\begin{eqnarray*}
\alpha_f(s) &: =&
 \big\vert E_f(s) 
 \big\vert, \, \mbox{ for } s>0,
\end{eqnarray*}
where $|A|$ denotes the Lebesgue measure of a set $A\subseteq \R^N.$
 Now, we define the {\it one dimensional decreasing rearrangement} $f^*$ of $f$ as below: 
 \begin{align*}
f^*(t):= \begin{cases*} \operatorname{ess}\ \sup f, \ \ t =0\\ \inf \{s>0 \, : \, \alpha_f(s) < t \}, \; t>0.   \end{cases*}                   
 \end{align*}
The map $f \mapsto f^*$ is not sub-additive. However, we obtain a sub-additive function from $f^*,$ namely the maximal function $f^{**}$ of $f^*$, defined by 
\begin{equation*}
f^{**}(t)=\frac{1}{t}\int_0^tf^*(\tau) d\tau, \quad t>0.
\end{equation*}
The sub-additivity of $f^{**}$ with respect to $f$ helps us to define norms in certain function spaces.

The { \it Schwarz symmetrization } of $f$ is defined by 
\begin{equation*}
  f^\star(x)=f^*(\omega_N|x|^N)  \label{relation}, \quad \forall\, x\in \Omega^\star,
\end{equation*}
where $\omega_N$ is the measure of the unit ball in $\R^N$ and $\Omega^\star$ is the open ball centered at the origin with same measure as $\Omega.$

Next, we state an important inequality concerning the Schwarz symmetrization, see Theorem 3.2.10 of \cite{EdEv}.
\begin{proposition}[Hardy-Littlewood inequality] 
Let $\Omega \subseteq \R^N$ with $N\ge 1$ and $f,g\in \M$ be nonnegative functions. Then
\begin{align} \label{HardyLittle}
 \int_{\Omega} f(x)g(x) \ dx \leq \int_{\Omega^\star}f^\star(x)g^\star(x) \ dx = \int_0^{|\Omega|}f^*(t) g^*(t)\ dt.
\end{align}
 \end{proposition} 

  \subsection{Banach function spaces}
\begin{definition} \label{BFC}
  A normed linear space $(X,\norm{.}_X)$ of functions in $\M$ is called a K\"othe function space if the following conditions are satisfied:
  \begin{enumerate}
   \item $\norm{f}_X = \| \ |f| \ \|_X$, for all $f \in X$,
   \item if $g_1 \in X$ and $|g_2| \leq |g_1|$ a.e., then $\|g_2\|_{X} \leq \|g_1\|_{X}$ and $g_2 \in X$.
  \end{enumerate}
  \end{definition}
  \noi The norm $\norm{.}_X$ is called a K\"othe function norm on $X.$ A complete K\"othe function space $(X,\norm{.}_X)$ is called as Banach function space and the associated norm $\norm{.}_X$ is called a Banach function space norm. 
  \begin{proposition} \cite[Theorem 2, Section 30, Chapter 6]{Zaanen} \label{BanachFS}
  Let $(X,\norm{.}_X)$ be a  K\"othe function space such that, for any non-negative sequence of function $ (f_n)$ in $X$ that increases to $f$, we have $\norm{f_n}_X$ increases to $ \norm{f}_X$. Then $(X,\norm{.}_X)$ is a Banach function space.  
  \end{proposition}
  For a Banach function space $(X,\norm{.}_X)$, we define its associate space as follows.
   \begin{definition} \label{assspacedef}
 Let $(X,\norm{.}_X)$ be a Banach function space. For $u\in \M$, define 
 \[ \norm{u}_{X'}= \sup \left\{ \int_{\Om} |fu|\ dx : f \in X, \ \norm{f}_X \leq 1 \right\}.\]
 Then the associate space $X'$ of $X$ is given by \[X' = \left\{ u\in \M\, :\norm{u}_{X'} < \infty \right \}.\]
\end{definition}
\noi One can verify that $X'$ is also a Banach function space with respect to the norm $\norm{.}_{X'}.$  For further readings on Banach function spaces, we refer to \cite{Bennett,EdEv,Zaanen}.

%
%

  \subsection{Lorentz spaces}
The Lorentz spaces are refinements of the usual Lebesgue spaces and introduced by Lorentz  in \cite{Lorentz}. For more details on Lorentz spaces and related results, we refer to the book \cite{EdEv}. 

 Given a function $f\in\mathcal{M}(\Omega)$ and $(p,q) \in [1,\infty)\times[1,\infty]$ 
we consider the following quantity:
\begin{align*} 
 |f|_{(p,q)} := \norm{t^{\frac{1}{p}-\frac{1}{q}} f^{*} (t)}_{{L^q((0,\infty))}}
=\left\{\begin{array}{ll}
         \left(\displaystyle\int_0^\infty \left[t^{\frac{1}{p}-\frac{1}{q}} {f^{*}(t)}\right]^q \ dt \right)^{\frac{1}{q}};\; 1\leq q < \infty, \vspace{4mm}\\ 
         \displaystyle\sup_{t>0}t^{\frac{1}{p}}f^{*}(t);\; q=\infty.
        \end{array} 
\right.
\end{align*}
The Lorentz space $L^{p,q}(\Om)$ is defined as
\[ L^{p,q}(\Om) := \left \{ f\in \mathcal{M}(\Om): \,   |f|_{(p,q)}<\infty \right \}.\]
$ |f|_{(p,q)}$ is  a complete quasi-norm on $L^{p,q}(\Om).$ 
For $(p,q) \in (1,\infty)\times[1,\infty]$, let  
 \[\norm{f}_{(p,q)}:= \norm{t^{\frac{1}{p}-\frac{1}{q}} f^{**} (t)}_{{L^q((0,\infty))}}.\]
Then $\norm{f}_{(p,q)}$ is a norm on $L^{p,q}(\Om)$ and it is equivalent to the quasi-norm $|f|_{(p,q)}$ (see Lemma 3.4.6 of \cite{EdEv}).
Next proposition identifies the associate space of Lorentz spaces,
 see \cite{Bennett} (Theorem 4.7, page 220).
 \begin{proposition} \label{assspace}
  Let $1<p<\infty$ and $1 \leq q\leq \infty$ (or $p=q=1$ or $p=q= \infty$). Then the associate space of $L^{p,q}(\Om)$ is, up to equivalence of norms, the Lorentz space $L^{p',q'}(\Om)$, where $\frac{1}{p} + \frac{1}{p'}=1$
 and $\frac{1}{q} + \frac{1}{q'}=1$.
  \end{proposition}
  
\subsection{The $p$-capacity} 
For any subset $A$ of $\R^N$ define,\[ {\cp(A)}:= \inf \left\{ \displaystyle \int_{\Omega} | \nabla u |^p\ dx : u \in \mathcal{D}^{1,p}_0(\R^N), A \subseteq \mbox{int} \{ u \geq 1\} \right\}.\]
It can be shown that the above definition is consistent with our earlier definition when $A$ is relatively compact  in $\Om$.  
 Next, we list some of the properties of capacity in the following proposition.
\begin{proposition} \label{propcap}
 \begin{enumerate}[(a)]
  \item Let $\Om_1 \subseteq \Om_2$ be open in $\R^N$. Then $\cp(., \Om_2) \leq \cp(., \Om_1).$
  \item $\cp$ is an outer measure on $\R^N.$
  \item  For $\lambda >0$ and $F \cset \R^N$, $ \cp(\lambda F ) = \lambda^{N-p} \cp( F ) $.
  \item For $F \cset \R^N$, $\exists C>0$ depending on $p, N$ such that $|F| \leq C \cp( F )^{\frac{N}{N-p}}$. 
  \item For $N>p,$ $\cp(B_1)= N \om_N \left(\frac{N-p}{p-1}\right)^{p-1},$ where $B_1$ is the unit ball in $\R^N$.
  \item $\cp(L(F))=\cp(F)$, for any affine isometry $L:\R^N \mapsto \R^N.$
\end{enumerate}
\end{proposition}

\begin{proof}
$(a)$ Follows easily from the definition of capacity.

\noi $(b)$ See Theorem 4.14 of \cite{Evans}(page 174).

 \noi $(c), (d),  (f)$ See Theorem 4.15 of \cite{Evans}(page 175).
 
\noi $(e)$ Section 2.2.4 of \cite{Mazya}  (page 106).
 \end{proof}
 
 \begin{rmk} \rm
If  a set $A$ is measurable with respect to $\cp$ then $\cp(A)$ must be $0$ or $\infty$, see Theorem 4.14 of \cite{Evans} (page 174).
\end{rmk}

 The next theorem follows from Maz'ya's characterization of Hardy potential, \cite{Mazya} (see Section 2.3.2, page 111).
\begin{thm} \label{Mazya's condition} 
Let $p \in (1,N)$, $\Om \subseteq \R^N$ be open and $g \in L^1_{loc}(\Om)$. If $g \in \H$, then
\[\int_{\Om} g|u|^p \ dx \leq C_H \norm{g} \int_{\Om} |\nabla u|^p \ dx, \forall u \in \Dp,\]
where $C_H=p^p(p-1)^{1-p}$.
\end{thm}
 It is easy to see that, the best constant satisfies the following inequalities: \begin{equation}\label{estimate}
     \norm{g} \leq B_g \leq C_H \norm{g}.
 \end{equation}

\subsection{The space of measures}
Let $\mathbb{M} (\R^N)$ be the space of all bounded signed measures on $\R^N.$ Then $\mathbb{M} (\R^N)$ is a Banach space with respect to the norm $\norm{\mu}=|\mu|(\R^N)$ (total variation of the measure $\mu$). 
The next proposition follows from the uniqueness of the Riesz representation theorem.
\begin{proposition} \label{defmeasure}
 Let $\mu \in \mathbb{M}(\R^N)$ be a positive measure. Then for an open $V \subseteq \R^N$,
 \[ \mu(V)= \sup \left \{ \int_{\R^N} \phi \ d\mu : 0 \leq \phi \leq 1, \phi \in C_c^{\infty}(\R^N) \ with \ Supp(\phi) \subseteq V   \right \}\]
and for any Borel set $E \subseteq \Om$, $\mu(E):= \inf \left\{ \mu(V) : E \subseteq V, \, V \text{open} \right\}$.
\end{proposition}

Recall that, a  sequence $(\mu_n)$ is said to be weak* convergent to $\mu$ in $\mathbb{M}(\R^N)$, if
 \begin{eqnarray*}
  \int_{\R^N} \phi \ d\mu_n \ra  \int_{\R^N} \phi \ d\mu, \ as \ n \ra \infty,  \forall\,  \phi \in C_c(\R^N).
 \end{eqnarray*}
 In this case we denote $\mu_n \wrastar \mu$. The next proposition is a consequence of Banach-Alaoglu theorem which states that for any Banach space $X$, the closed  unit ball in  $X^*$ is weak* compact.
\begin{proposition} \label{BanachAlaoglu}
 Let $(\mu_n)$ be a bounded sequence in $\mathbb{M}(\R^N)$, then there exists $\mu \in \mathbb{M}(\R^N)$ such that $\mu_n \overset{\ast}{\rightharpoonup} \mu$ up to a subsequence.
\end{proposition}

\subsection{Brezis-Lieb lemma}
The following lemma is due to Brezis and Lieb (see Theorem 1 of \cite{Brezis-Lieb}). 
\begin{lem}\label{Bresiz-Lieb}
 Let $(\Om, \A, \mu)$ be a measure space and $(f_n)$ be a sequence of complex
 -valued measurable functions which are uniformly bounded in $L^p(\Om, \mu)$ for some $0<p< \infty$. Moreover, if $(f_n)$ converges to $f$ a.e., then
\[\lim_{n \ra \infty} \left| \norm{f_n}_{(p, \mu)} - \norm{f_n-f}_{(p, \mu)} \right| = \norm{f}_{(p, \mu)}.\]
\end{lem}
\noi We also  require the following inequality (see \cite{Lieb}, page 22) that played an important role in the proof of Bresiz-Lieb lemma: for $a,b \in \mathbb{C}$,
\begin{equation} \label{inequality}
\big||a+b|^p-|a|^p \big|\leq \epsilon |a|^p + C(\epsilon,p)|b|^p 
\end{equation}
valid for each $\epsilon >0$ and  $0< p< \infty$.

 \section{Embeddings} \label{Em}
 In this section we prove the following continuous embeddings: 
 \[L^{{\frac{N}{p}},\infty}(\Om)\hookrightarrow \H; \quad \F_{\frac{N}{p}}(\Om) \hookrightarrow \F(\Om);\quad  I(\Om)\hookrightarrow \F(\Om).\]
 We provide alternate proofs for certain classical embeddings and also provide an embedding of $\Dp$ finer than Lorentz-Sobolev embeddings.
\begin{proposition}
 For  $p \in (1,N)$ and an open subset $\Om$  in  $\R^N$, $L^{{\frac{N}{p}},\infty}(\Om)$ is continuously embedded in $\H$.
 \end{proposition}
  \begin{proof}
  Observe that, $\cp(F^{\star}) \leq \cp(F^{\star}, \Om^{\star}) \leq  \cp(F, \Om) $. The first inequality comes from $(a)$-th property of Proposition \ref{propcap} and the latter one follows from Polya-Szego inequality.
  $\cp(F^{\star})= N \om_N (\frac{N-p}{p-1})^{p-1}R^{N-p},$ where $R$ is the radius of $F^{\star}$ (by $(e)$-th property of Proposition \ref{propcap}).
  Now, for a relatively compact set $F$,
  \begin{eqnarray*}
  \frac{\int_F |g|(x)\ dx}{\cp(F, \Om)} &\leq& \frac{\int_{F^{\star}} g^{\star}(x) \ dx}{\cp(F^{\star}, \R^N)} = \frac{\int_0^{|F|} g^*(t) \ dt}{N \om_N (\frac{N-p}{p-1})^{p-1}R^{N-p}} =
  \frac{R^p g^{**}(\om_N R^{N})}{N (\frac{N-p}{p-1})^{p-1}}.
  \end{eqnarray*}
  By setting $\om_N R^{N} =t$ we get,
  \begin{eqnarray*}
  \frac{\int_F |g|(x) \ dx}{\cp(F, \Om)} \leq C(N,p) \norm{g}_{(\frac{N}{p}, \infty)}.
  \end{eqnarray*}
  Now take the supremum over $F \cset \Om$ to obtain, 
 \[\norm{g} \leq C(N,p)  \norm{g}_{(\frac{N}{p}, \infty)} \text{ with } C(N,p)=\frac{1}{N (\om_N)^{\frac{p}{N}} (\frac{N-p}{p-1})^{p-1}}.\]
  \end{proof}

\begin{proposition} \label{Iomega}
  For  $p \in (1,N)$ and an open subset $\Om$  in  $\R^N$, $I(\Om)$ is continuously embedded into $\H$.
  \end{proposition}
  \begin{proof}
   For $g \in I(\Om)$ and $ u \in \mathcal{N}(F)$, use Lemma 2.1 of \cite{ADS-exterior} to obtain
\[ \int_{F} |g|\ dx \leq  \int_{\Omega} |g||u|^p dx \leq C_H \norm{g}_I \int_{\Omega} |\nabla u|^p dx , \ \forall u \in \Dp,\]
  where $C$ depends only on $N, p.$ Taking the infimum over $\mathcal{N}(F)$ and then the supremum over $F$ we obtain $\norm{g} \leq C_H \norm{g}_I$.
  \end{proof}
  
  \begin{rmk} \rm
  Notice that $I(\Om)$ and $\H$ are not rearrangement invariant Banach function spaces. For example,  for $p=2$, $N \geq 3$ and 
$\be \in (\frac{2}{N},1)$ consider the following function analogous to Example 3.8 of \cite{biharmonic}, 
\begin{equation*}
      g(x)= \left\{
          \begin{array}{ll}
          (|x|-1)^{-\be}, & 1<|x|\leq 2, \\
                                                         0 , & {\mbox{otherwise}}.
                                                          \end{array}\right .
  \end{equation*}
 It can be verified that $g \in I(\R^N)$ and $g^{\star}\notin \mathcal{H}(\R^N)$.
  \end{rmk}
  
  As we have mentioned before, if $g \in \F_{\frac{N}{p}}(\Om)$ then $G$ is compact and the same is true if $g \in I(\Om)$, where $\Om = \overline{B_1}^c$. Next proposition shows that $\F(\Om)$ contains these spaces.
\begin{proposition}
  Let  $p \in (1,N)$. Then
  \begin{enumerate}[(i)]
      \item  $\F_{\frac{N}{p}}(\Om) \subseteq \F(\Om)$ for any open subset $\Om$  in  $\R^N$.
      \item $I(\Om)\subseteq \F(\Om)$ for $\Om= B_d \setminus \overline{B_c}$; $0 \leq c <d \leq \infty$.
  \end{enumerate}
 \end{proposition}
 \begin{proof}
Recall that, $\F_{\frac{N}{p}}(\Om)$ is the closure of $C_c^{\infty}(\Om)$ in $L^{{\frac{N}{p}},\infty}(\Om)$ and $\F(\Om)$ is closure of $C_c^{\infty}(\Om)$ in $\H.$ Now since $\norm{.} \leq C \norm{.}_{({\frac{N}{p}},\infty)}$, 
it is immediate that $\F_{\frac{N}{p}}(\Om)$ is contained in $\F(\Om)$. 
Similarly, in order to prove $(ii)$, it is enough to show $C_c^{\infty}(\Om)$ is dense in $I(\Om)$. For this, let $g \in I(\Om)$ and $\epsilon >0$ be arbitrary.
 As $C_c^{\infty} ((c, d))$ is dense in $L^1((c, d), r^{p-1})$,
there exists $ \phi \in C_c^{\infty}((c, d))$ such that $\norm{\tilde{g} - \phi}_{L^1((c, d),r^{p-1})} < \epsilon $.
 Now, for $x \in \Om$ let $\si(x) := \phi(|x|) $.  By denoting, $h= g- \si$ we have, $ \tilde{h}(r) = \tilde{g}(r) - \phi(r)$.
 Therefore,
 \[\norm{g - \si}_I= \int_0^{\infty}  |\tilde{h}|(r) r^{p-1} dr = \norm{\tilde{g} - \phi}_{L^1((0,\infty),r^{p-1})} < \epsilon.\]
  \end{proof}
 \begin{rmk} \rm In \cite[Lemma 3.5]{anoop-p}, authors have shown that $\F_{\frac{N}{p}}(\Om)$ contains the Hardy potentials that have faster decay than $\frac{1}{|x-a|^p}$ at all points $a \in \overline{\Om}$ and at infinity. Such Hardy potentials arises in the work of Szulkin
and Willem \cite{Andrze}. Above proposition assures that they belong to $\F(\Om)$.
  \end{rmk} 
  
  Next, we give an alternate proof for the Lorentz-Sobolev embedding of $\Dp$. The idea is similar to that of Corollary 3.6 of \cite{biharmonic}.
  \begin{proposition} \label{embedding}
  For  $p \in (1,N)$ and an open subset $\Om$  in  $\R^N$,  $\Dp$ is continuously embedded in $L^{p^*,p}(\Om).$ 
\end{proposition}
\begin{proof}
 Without loss of generality we may assume $\Om = \R^N$ (for a  general domain $\Om$, the result will follow by considering the zero extension to $\R^N$). Let $g \in \M $ be such that $g^{\star} \in \M. $ 
 Then using the Polya-Szego inequality we have,
  \[ \int_{\R^N} g^{\star} |u^{\star}|^p \ dx \leq C_H \norm{g^{\star}} \int_{\R^N} |\nabla u^{\star}|^p \ dx \leq C_H \norm{g^{\star}} \int_{\R^N} |\nabla u|^p \ dx , \ \forall u \in \D^{1,p}_0(\R^N).\]
    In particular, for $g(x) = \frac{1}{\om_N^{\frac{p}{N}} |x|^p}$, $g^*(s) = \frac{1}{s^{\frac{p}{N}}}$ and $\norm{g^{\star}} = \frac{(p-1)^{p-1}}{N (N-p)^{p-1}} $.
    Now $\int_{\R^N}g^{\star} |u^{\star}|^p \ dx = \int_0^{\infty} g^*(s) |u^*(s)|^p \ ds $. Thus
    from the above inequality we obtain, 
    \[\int_0^{\infty} \frac{|u^*(s)|^p}{s^{\frac{p}{N}}} \ ds \leq C(N,p)  \int_{\R^N} |\nabla u|^p \ dx, \, \forall\, u\in \Dp. \]  
    The left hand side of the above inequality is $ |u|_{(p^*,p)}^p$, a quasi-norm  equivalent to the norm $\norm{u}^p_{(p^{*},p)}$ in $L^{p^*,p}(\Om)$. This completes the proof.  
    \end{proof}
 \begin{corollary} \label{cpctembedding}
      For  $p \in (1,N)$ and an open subset $\Om$  in  $\R^N$, $\Dp$ is compactly embedded in $L^{p}_{loc}(\Om).$
    \end{corollary}
\begin{proof}
  Clearly $\Dp$ is continuously embedded into $W^{1,p}_{loc}(\Om)$. Since $W^{1,p}_{loc}(\Om)$ is compactly embedded in $L^{p}_{loc}(\Om),$ we have the required embedding.
\end{proof}

Notice that we used just one Hardy potential $\frac{1}{|x|^p}$ to obtain the Lorentz-Sobolev embedding in Proposition \ref{embedding}. Instead, if we consider all of $\H$, then we anticipate to get an embedding finer than 
the above one. For this, we consider the following space (defined in a similar way as the associate space):
  \[\E(\Om):= \left \{ u\in \M : |u|^p  \in \H' \right \} .\]
 One can verify that $\E(\Om)$ is a Banach function space with respect to the norm 
  \begin{eqnarray*}
   \norm{u}_{\E}:= \left( \norm{|u|^p}' \right)^{\frac{1}{p}}.
  \end{eqnarray*}

 In the next theorem, we establish an embedding of $\Dp$ into $\E(\Om)$. Further, we assert that the embedding is finer than the classical one.
  \begin{thm} \label{opembedding}
 Let $1<p<N$ and $\Om$ be open in $\R^N$. Then
  \begin{enumerate}[(a)]
      \item $\Dp$ is continuously embedded into $\E(\Om)$,
      \item $\E(\Om) $ is a proper subspace of $L^{p^*,p}(\Om). $ 
  \end{enumerate} 
  \end{thm}
  \begin{proof}
  \begin{enumerate}[(a)]
      \item For $g \in \H$, by Theorem \ref{Mazya's condition},
\[ \int_{\Om} g|u|^p \ dx \leq C_H \norm{g} \int_{\Om} |\nabla u|^p \ dx, \forall u \in \Dp.\]
  Now taking the supremum over  the unit ball in $\H$ we obtain,
  \[ \norm{u} _{\E} \leq C_H^{\frac{1}{p}} \norm{u}_{\Dp}, \forall u \in \Dp.\]
  \item Clearly $v \in \E(\Om)$ if and only if $|v|^p \in \H'$. Further, $L^{{\frac{N}{p}},\infty}(\Om) \subsetneq \H$ and hence $\H' \subsetneq L^{\frac{p^*}{p},1}(\Om)$ (by Proposition \ref{assspace}). 
Now, we can easily deduce  that $\E(\Om) \subsetneq L^{p^*,p}(\Om).$
  \end{enumerate}
  \end{proof}

\section{The compactness} \label{Compactness}

In this section, we develop a $g$ depended concentration compactness lemma as in \cite{Smets}.  Then we give  equivalent conditions for compactness and prove Theorem \ref{eqivthm}, Theorem \ref{eqivthm1} and Theorem \ref{eqivthm2}. 
\begin{lem} \label{uineq}
  Let $\Phi \in C_b^1(\Om)$ be such that $\nabla \Phi$ has compact support and $u_n \wra u$ in $\Dp$. Then
  \[ \overline{\lim_{n \ra \infty}} \int_{\Om} |\nabla((u_n -  u)\Phi)|^p \ dx= \overline{\lim_{n \ra \infty}} \int_{\Om} |\nabla(u_n -  u)|^p |\Phi|^p \ dx.\]
 \end{lem}
 \begin{proof}
  Let $\epsilon >0$ be given. Using \eqref{inequality},
   \begin{eqnarray*}
      \bigg| \int_{\Om} |\nabla((u_n  &-&  u)\Phi)|^p\ dx - \int_{\Om} |\nabla(u_n -  u)|^p |\Phi|^p \ dx\bigg|  \\
              &\leq&   \epsilon \int_{\Om} |\nabla(u_n -  u)|^p |\Phi|^p\ dx + C(\epsilon,p) \int_{\Om} |u_n -  u|^p |\nabla \Phi|^p\ dx. 
   \end{eqnarray*}
 Since $\nabla \Phi$ is compactly supported, by Corollary \ref{cpctembedding} the second term in the right-hand side of the above inequality goes to 0 as $n \ra \infty$. Further, as $(u_n)$ is bounded in $\Dp$ and $\epsilon>0$ 
 is arbitrary, we obtain the desired result.
 \end{proof}

A function in $\Dp$ can be considered as a function in $\mathcal{D}^{1,p}_0(\R^N)$ by usual zero extension. Following this convention,  for $u_n, u\in \Dp$ and a Borel set  $ E $ in   $\R^N,$ we  denote \[ \nu_n(E)=\int_E g|u_n - u|^p \ dx;\qquad \Ga_n(E)=\int_E|\nabla(u_n-u)|^p\ dx;\qquad \widetilde{\Ga}_n(E)=\int_E|\nabla u_n|^p\ dx.\] If $u_n\wra u$ in $\Dp$, then $\nu_n$, $\Gamma_n$ and $\widetilde{\Gamma}_n$
  have weak* convergent sub-sequences (Proposition \ref{BanachAlaoglu}). Let
   \[\nu_n \wrastar \nu; \qquad \Gamma_n \overset{\ast}{\rightharpoonup} \Gamma; \qquad \widetilde{\Gamma}_n \overset{\ast}{\rightharpoonup} \widetilde{\Gamma} \  \mbox{ in } \mathbb{M}(\R^N).\]

 Next, we prove the absolute continuity of the measure $\nu$ with respect to $\Gamma$. 
 \begin{lem}\label{mlc1}
 Let $g \in \H $, $g \geq 0$ and $u_n \wra u$ in $ \Dp$.  Then  for any Borel set $E$ in $\R^N$,
   \[\nu(E) \leq C_H \ \C^*_g \Gamma (E), \ \mbox{where $\C^*_g = \displaystyle\sup_{x \in \overline{ \Om}} \C_g(x)$ } .\]
  \end{lem}
  \begin{proof}
As $u_n \wra u$ in $ \Dp$, $u_n \ra u$ in $L^{p}_{loc}(\Om)$ (by Corollary \ref{cpctembedding}).  For $\Phi \in C_c^{\infty}(\R^N)$, $(u_n-u) \Phi \in \D^{1,p}_0(\Om)$ and thus by Theorem \ref{Mazya's condition},
\begin{eqnarray*}
 \int_{\R^N} |\Phi|^p \ d \nu_n  = \int_{\Om} g|(u_n-u)\Phi|^p \ dx & \leq &  C_H \norm{g}  \int_{\Om} |\nabla((u_n -  u)\Phi)|^p \ dx  \\
& = & C_H \norm{g}  \int_{\R^N} |\nabla((u_n -  u)\Phi)|^p \ dx.
\end{eqnarray*}
Take $n \ra \infty$ and use Lemma \ref{uineq} to obtain
 \begin{align} \label{forrmk}
     \int_{\R^N} |\Phi|^p \ d \nu \leq C_H \norm{g} \int_{\R^N} |\Phi|^p \ d \Gamma .
 \end{align}
Now by Proposition \ref{defmeasure}, we get 
  \begin{equation}\label{measureinequality1}
   \nu(E) \leq C_H  \norm{g} \Gamma (E)  \ , \forall E \ Borel \ in \ \R^N.
  \end{equation}
 In particular, $\nu \ll \Gamma$ and hence by Radon-Nikodym theorem, 
 \begin{equation} \label{measureinequality}
  \nu(E) = \int_E  \frac{d \nu}{d \Gamma} \  d\Gamma \ , \forall E \ Borel \ in \ \R^N. 
 \end{equation}
 Further, by Lebesgue differentiation theorem (page 152-168 of \cite{Federer}) we have 
 \begin{equation} \label{Lebdiff}
  \frac{d \nu}{d \Gamma}(x) = \lim_{r \ra 0} \frac{\nu (B_r(x))}{\Gamma (B_r(x))}.
 \end{equation}
 Now replacing $g$ by $g \chi_{B_r(x)}$ and proceeding as before,
 \[ \nu(B_r(x)) \leq C_H \norm{g \chi_{B_r(x)}} \ \Gamma (B_r(x)).\] 
 Thus from \eqref{Lebdiff} we get 
\begin{eqnarray} \label{21}
 \frac{d \nu}{d \Gamma} (x) \leq C_H \C_g(x)
\end{eqnarray} 
and hence 
$\norm{\frac{d \nu}{d \Gamma}}_{\infty} \leq C_H \C^*_g$. Now from \eqref{measureinequality} we obtain $\nu(E) \leq C_H \ \C^*_g \Gamma (E)$ 
for all Borel subsets $E$ of $\R^N$.
\end{proof} 
\begin{rmk} \label{singularset} \rm
 In \cite{Tertikas} (for $p=2$ and  $\Om=\R^N$) and in \cite{Smets} (for  $p\in (1,N)$ and $\Om \subseteq \R^N$), the authors considered the following concentration function: 
\begin{eqnarray*}
 S_{g}(x) & = & \lim_{r\ra 0} \inf \left \{\int_{\Om} |\nabla u|^p \ dx : u\in \D^{1,p}_0 (\Om \cap B_r(x)), \ \int_{\Om} g|u|^p \ dx =1 \right \}, 
 \end{eqnarray*}
and they considered the singular set to be $\left\{x \in \overline{\Om}: S_g(x)< \infty \right\}$ and assumed that the closure of it, is  at most countable (see (H) of \cite{Tertikas} and (H1) of \cite{Smets}). One can easily see that their singular set coincides with $\sum_g$ (by \eqref{S_gC_g}). 
The countability assumption allowed them to  describe $\nu$ as a countable sum of Dirac measures located  on $\sum_g$ and using this they have obtained the absolute continuity of $\nu$ with respect to $\Gamma$ (see Lemma 2.1 of \cite{Smets} and Lemma 3.1 of \cite{Tertikas}). Whereas we use the Radon-Nikodym theorem  and the Lebesgue differentiation theorem to prove the absolute continuity of $\nu$ with respect to $\Gamma$. We would like to stress that we do not make any assumption  on the cardinality or the structure of $\sum_g$ for this purpose.
 \end{rmk}

The next lemma gives a lower estimate for the measure $\tilde{\Gamma}.$ Similar estimate is obtained in Lemma 2.1 of \cite{Smets}.  We make a weaker assumption, $\overline{\sum_g}$ is of Lebesgue measure $0$, than the assumption $\overline{\sum_g}$ is countable.
\begin{lem} \label{lemma9}
Let $g\in \H$ be such that $g\ge0$ and $|\overline{\sum_g}|=0$. If $u_n \wra u$ in $ \Dp$, then 
\begin{equation*}
   \tilde{\Gamma} \geq \begin{cases}
     |\nabla u|^p + \frac{\nu}{C_H \C_g^*}, \quad \text{if} \ \C_g^* \neq 0, \\
     |\nabla u|^p, \quad \text{otherwise}.
    \end{cases}
\end{equation*}
\end{lem}
\begin{proof}
Our proof splits in to three steps.\\
{\bf Step 1:} $\tilde{\Ga} \geq |\nabla u|^p.$ Let $\phi \in C_c^{\infty}(\R^N)$ with $0\leq \phi \leq 1$, we need to show that
$\int_{\R^N} \phi \ d\tilde{\Ga} \geq \int_{\R^N} \phi |\nabla u|^p \ dx. $
Notice that,
$$\int_{\R^N} \phi \ d\tilde{\Ga}= \lim_{n \ra \infty} \int_{\R^N} \phi \ d\tilde{\Ga}_n  = \lim_{n \ra \infty} \int_{\Om} \phi |\nabla u_n|^p \ dx = \lim_{n \ra \infty} \int_{\Om} F(x,\nabla u_n(x)) \ dx ,$$
where $F:\Om \times \R^N \mapsto \R$ is defined as $F(x,z)=\phi(x)|z|^p.$ Clearly, $F$ is a Caratheodory function  and $F(x,.)$ is convex for almost every $x$. Hence, by 
Theorem 2.6 of \cite{Fillip} (page 28), we have
$\displaystyle \lim_{n \ra \infty} \int_{\Om} \phi |\nabla u_n|^p \ dx \geq  \int_{\Om} \phi |\nabla u|^p \ dx= \int_{\R^N} \phi |\nabla u|^p \ dx$ and this proves our claim 1.

\noi {\bf Step 2:} $\tilde{\Gamma}=\Gamma$, on $\overline{\sum_g}.$ Let $E\subset\overline{\sum_g}$ be a Borel set. Thus, for each $m \in \N$, there exists an open subset $O_{m}$ containing $E$ such that $|O_m|=|O_{m} \setminus E| < \frac{1}{m}$. Let $\var >0$ be given. Then, for any $\phi \in C_c^{\infty}(O_{m})$ with $0 \leq \phi \leq 1$, using \eqref{inequality} we have
\begin{align*}
    \left|\int_{\Om}  \phi \ d\Gamma_n \ dx -\int_{\Om}  \phi \ d\tilde{\Gamma}_n \ dx \right|&= \left|\int_{\Om}  \phi |\nabla (u_n-u)|^p \ dx -\int_{\Om}  \phi |\nabla u_n|^p \ dx \right| \\
    &\leq \var \int_{\Om}  \phi |\nabla u_n|^p \ dx + \text{C}(\var,p) \int_{\Om} \phi |\nabla u|^p dx \\
 & \leq   \var L + \text{C}(\var,p) \int_{O_{m}}   |\nabla u|^p \ dx,
\end{align*}
where $L=\sup_{n}\left\{\int_{\Om} |\nabla u_n|^p \ dx\right\}$. 
Now letting $n \ra \infty$, we obtain $
 \left|\int_{\Om}  \phi \ d\Gamma  -\int_{\Om}  \phi \ d\tilde{\Gamma}  \right|  \leq \var L + \text{C}(\var,p) \int_{O_{m}}   |\nabla u|^p \ dx.
$
Therefore,
\begin{eqnarray*} 
\left|\Gamma (O_m)-\tilde{\Gamma} (O_m) \right| &=& \sup \left\{ \left|\int_{\Om}  \phi \ d\Gamma -\int_{\Om}  \phi \ \ d\tilde{\Gamma} \right|: \phi \in C_c^{\infty}(O_m), 0 \leq \phi \leq 1 \right \} \\
&\leq & \var L + \text{C}(\var,p) \int_{O_{m}}   |\nabla u|^p \ dx,
\end{eqnarray*}
Now as $m \ra \infty,$ $|O_m|\ra 0$ and hence $| \Gamma (E)-\tilde{\Gamma} (E)| \leq \var L.$ Since $\var >0$ is arbitrary, we conclude $\Gamma(E)=\tilde{\Gamma} (E).$

\noi {\bf{Step 3:}} $ \tilde{\Gamma} \geq |\nabla u|^p + \frac{\nu}{C_H \C_g^*},$ if $\C_g^* \neq 0$. Let $\C_g^* \neq 0$. Then from Lemma \ref{mlc1} we have  $ \Gamma \geq \frac{\nu}{C_H \C_g^*}$. Furthermore, \eqref{21} and \eqref{measureinequality} ensures that $\nu$ is supported on $\sum_g.$ Hence Step 1 and Step 2 yields the following: 
\begin{equation}\label{rep1}
 \tilde{\Gamma} \geq \left\{\begin{array}{ll}
    |\nabla u|^p,  &  \\
     \frac{\nu}{C_H \C_g^*} .
\end{array}\right.    
\end{equation}
Since $|\overline{\sum_g}|=0$, the measure $|\nabla u|^p$ is supported inside $\overline{\sum_g}^{c}$ and hence from \eqref{rep1} we easily obtain  $\tilde{\Gamma} \geq |\nabla u|^p +  \frac{\nu}{C_H \C_g^*}.$
\end{proof}

\begin{lem}\label{lemmasmets}
  Let $g \in \H$, $g \geq 0$ and $u_n \wra u$ in $ \Dp$ and $\Phi_R \in C_b^{\infty}(\R^N)$ with $0\leq \Phi_R \leq 1$, $\Phi_R = 0 $ on $\overline{B_R}$ and $\Phi_R = 1 $ on $B_{R+1}^c$.
Then,
\begin{enumerate}[(A)]
 \item $\displaystyle\lim_{R \ra \infty} \overline{ \lim_{n \ra \infty}}  \int_{\Om \cap \overline{B_R}^c} g|u_n|^p \ dx =\lim_{R \ra \infty} \overline{\lim_{n \ra \infty}}  \nu_n(\Om \cap \overline{B_R}^c) =\lim_{R \ra \infty} \overline{\lim_{n \ra \infty}}  \int_{\Om}  \Phi_R \ d \nu_n,$
\item $\displaystyle \lim_{R \ra \infty} \overline{\lim_{n \ra \infty}}  \int_{\Om \cap \overline{B_R}^c} |\nabla u_n|^p \ dx =\lim_{R \ra \infty} \overline{\lim_{n \ra \infty}}  \Gamma_n(\Om \cap \overline{B_R}^c)=\lim_{R \ra \infty} \overline{\lim_{n \ra \infty}}  \int_{\Om}  \Phi_R \ d \Gamma_n.$
\end{enumerate}                                                                               
 \end{lem}
\begin{proof}
 By Brezis-Lieb lemma,
 \begin{align*}
 \displaystyle \overline{\lim_{n \ra \infty}} \left| \nu_n(\Om \cap \overline{B_R}^c) - \int_{\Om \cap \overline{B_R}^c} g|u_n|^p \ dx \right| &= 
 \overline{\lim_{n \ra \infty}} \left|\int_{\Om \cap \overline{B_R}^c} g|u_n-u|^p \ dx - \int_{\Om \cap \overline{B_R}^c} g|u_n|^p \ dx \right| \\ &= \int_{\Om \cap \overline{B_R}^c} g|u|^p \ dx.
  \end{align*}
 As $g|u|^p \in L^1(\Om)$, the right-hand side integral goes to $0$ as $R \ra \infty$. Thus, we get the first equality in $(A)$. For the second equality, it is enough to observe that
 \begin{eqnarray*}
  \int_{\Om \cap \overline{B_{R+1}}^c} g|u_n -u|^p \ dx \leq \int_{\Om} g|u_n -u|^p \Phi_R \ dx \leq \int_{\Om \cap \overline{B_{R}}^c} g|u_n -u|^p \ dx.
 \end{eqnarray*}
Now by taking $n, R \ra \infty$ respectively we get the required equality. Now we proceed to prove (B). For $\var >0$, there exists $ \text{C}(\var,p)>0$ (by \eqref{inequality}) such that
\begin{align*}
& \, \displaystyle \overline{\lim_{n \ra \infty}} \left| \Gamma_n(\Om \cap \overline{B_R}^c) - \int_{\Om \cap \overline{B_R}^c} |\nabla u_n|^p \ dx \right| \\ &= 
\overline{\lim_{n \ra \infty}} \left|\int_{\Om \cap \overline{B_R}^c} |\nabla (u_n-u)|^p \ dx - \int_{\Om \cap \overline{B_R}^c} |\nabla u_n|^p \ dx \right| \\ & \leq \var \overline{\lim_{n \ra \infty}} \int_{\Om \cap \overline{B_R}^c}   |\nabla u_n|^p \ dx + \text{C}(\var,p) \int_{\Om \cap \overline{B_R}^c} |\nabla u|^p \ dx \\
& \leq \var L + \text{C}(\var,p) \int_{\Om \cap \overline{B_R}^c} |\nabla u|^p \ dx \,,
\end{align*}
where $L \geq \int_{\Om} |\nabla u_n|^p \ dx$ for all $n$.
Thus, by taking $R \ra \infty$ and then $\var \ra 0$, we obtain the first equality of (B). The second equality of part (B) follows from the same argument as that of part (A). 
\end{proof}
 
  \begin{lem}\label{mlc2}
   Let $g \in \H$, $g \geq 0$ and $u_n \wra u$ in $\Dp$. Set
   \begin{eqnarray*}
    \nu_{\infty} = \displaystyle\lim_{R \ra \infty} \overline{\lim_{n \ra \infty}}  \nu_n(\Om \cap \overline{B_R}^c) \quad \mbox{and} \quad \Gamma_{\infty} =  \displaystyle\lim_{R \ra \infty} \overline{\lim_{n \ra \infty}}  \Gamma_n(\Om \cap \overline{B_R}^c).
   \end{eqnarray*} 
  Then
\begin{enumerate}[(i)]
 \item $\nu_{\infty} \leq C_H\ \C_g(\infty)  \Gamma_{\infty} \nonumber,$
  \item $ \displaystyle \overline{\lim}_{n \ra \infty} \int_{\Om} g|u_n|^p \ dx = \int_{\Om} g|u|^p \ dx + \norm{\nu} + \nu_{\infty}.$
\item[(iii)] Further, if $|\overline{\sum_g}|=0$, then we have

\begin{equation*}
\displaystyle \overline{\lim}_{n \ra \infty} \int_{\Om} |\nabla u_n|^p \ dx  \geq \begin{cases}
\displaystyle \int_{\Om} |\nabla u|^p \ dx + \frac{\norm{\nu}}{C_H\C_g^*} +  \Gamma_{\infty}, \quad \ \text{if} \ \ \C_g^* \neq 0 \\
\displaystyle \int_{\Om} |\nabla u|^p \ dx +  \Gamma_{\infty}, \quad \ \text{otherwise}. 
     \end{cases}
  \end{equation*}
\end{enumerate}
  \end{lem}
 \begin{proof}
 (i): For $R>0$, choose $\Phi_R \in C_b^{1}(\R^N)$ satisfying  $0\leq \Phi_R \leq 1$, $\Phi_R = 0 $ on $\overline{B_R}$ and $\Phi_R = 1 $ on $B_{R+1}^c$. Clearly, $(u_n-u) \Phi_R \in \D^{1,p}_0(\Om \cap \displaystyle\overline{B_R}^c)$. Since
$\norm{g\chi_{ \overline{B_R}^c}} < \infty, $ by Theorem \ref{Mazya's condition},
\[ \int_{\Om \cap \overline{B_R}^c} g|(u_n-u)\Phi_R|^p \ dx \leq C_H \  \norm{g \chi_{\overline{B_R}^c}}  \int_{\Om \cap\overline{B_R}^c} |\nabla((u_n -  u)\Phi_R)|^p \ dx . \]
By Lemma \ref{uineq} we have, $\displaystyle \overline{\lim_{n \ra \infty}} \int_{\Om \cap\overline{B_R}^c} |\nabla((u_n -  u)\Phi_R)|^p \ dx =  \overline{\lim_{n \ra \infty}} \int_{\Om \cap\overline{B_R}^c} | \Phi_R|^p \ d \Gamma_n$. 
Therefore, letting $n \ra \infty$, $ R \ra \infty$ and using Lemma \ref{lemmasmets} successively in the above inequality we obtain $ \nu_{\infty} \leq C_H  \C_g(\infty) \ \Gamma_{\infty}.$ 

\noi $(ii)$: By choosing $\Phi_R$ as above and using Brezis-Lieb lemma together with part (A) of Lemma \ref{lemmasmets} we have,
\begin{align*}
& \, \overline{\lim_{n\ra \infty}} \int_{\Om} g| u_n|^p \ dx \\ &=  \overline{\lim_{n\ra \infty}} \left[ \int_{\Om} g| u_n|^p (1-\Phi_R) \ dx + \int_{\Om} g| u_n|^p \Phi_R \ dx  \right] \\
                                          &= \overline{\lim_{n\ra \infty}} \left[  \int_{\Om} g| u|^p (1-\Phi_R) \ dx  + \int_{\Om} g| u_n-u|^p (1-\Phi_R) \ dx + \int_{\Om} g| u_n|^p \Phi_R \ dx \right]   \\
                                          &= \int_{\Om} g| u|^p \ dx + \norm{\nu} + \nu_{\infty}. 
\end{align*} 
 
 \noi $(iii)$: Notice that 
 \begin{align*}
\overline{\lim_{n\ra \infty}} \int_{\Om} |\nabla u_n|^p \ dx &= \overline{\lim_{n\ra \infty}} \left[ \int_{\Om} |\nabla u_n|^p (1-\Phi_R)\ dx + \int_{\Om} |\nabla u_n|^p \Phi_R \ dx \right] \\ &= \tilde{\Gamma}(1-\Phi_R)+ \overline{\lim_{n\ra \infty}}  \int_{\Om} |\nabla u_n|^p \Phi_R \ dx
  \end{align*}
By taking $R \ra \infty$ and using part (B) Lemma \ref{lemmasmets} we get
\begin{equation*} 
 \overline{\lim_{n\ra \infty}} \int_{\Om} |\nabla u_n|^p \ dx = \|\tilde{\Gamma}\|+ \Gamma_{\infty}.   
\end{equation*}
Now, using Lemma \ref{lemma9}, we obtain \begin{equation*}
\displaystyle \overline{\lim}_{n \ra \infty} \int_{\Om} |\nabla u_n|^p \ dx \geq \begin{cases}
\displaystyle \int_{\Om} |\nabla u|^p \ dx + \frac{\norm{\nu}}{C_H\C_g^*} +  \Gamma_{\infty}, \quad \ \text{if} \ \C_g^* \neq 0 \\
\displaystyle \int_{\Om} |\nabla u|^p \ dx  +  \Gamma_{\infty}, \quad \ \text{otherwise}. 
     \end{cases}
  \end{equation*}
 \end{proof}

\begin{rmk}[ The assumptions on the  the singular set $\sum_g$] \rm
Let us recall the following fundamental result in the concentration compactness theory  by Lions \cite[Lemma 1.2]{Lions2a}.
Let $\nu, \Gamma$ be two non-negative, bounded measures on $\R^N$ such that
 \begin{equation} \label{lionsineq}
 \left[\int_{\R^N} |\phi|^q \ d\nu \right]^{\frac{1}{q}} \leq C \left[ \int_{\R^N} |\phi|^p \ d\Gamma \right]^{\frac{1}{p}}, \ \forall \,\phi \in C_c^{\infty}(\R^N),
 \end{equation}
 for some $C>0$ and $1 \leq p < q.$ Then there exist at most a countable set $\left\{x_j \in \R^N : j \in \mathbb{J} \right\}$ and $ \nu_j \in (0, \infty)$ such that
 \begin{align} \label{lionseq}
     \nu = \displaystyle \sum_{j \in \mathbb{J}} \nu_j \delta_{x_j}.
 \end{align}
 For $q=p$, in  \cite{Smets,Tertikas} authors assumed the countability of the singular set  $\overline{\sum_g}$ and obtain the  same  representation of $\gamma$ as in \eqref{lionseq}. This representation helps them for proving  the results (\cite[Lemma 3.1]{Tertikas}, \cite[Lemma 2.1]{Smets}). In this situation, we have seen that $\gamma$ is supported on the set $\overline{\sum_g}$ (by Lemma \ref{mlc1}). In this article, we relax the countability assumption on $\overline{\sum_g}$ and by pass the representation of $\gamma$ in order to derive Lemma \ref{mlc2}.  Indeed,  we have \eqref{forrmk} which is the limiting case of \eqref{lionsineq} ($q=p$).
\end{rmk}

In the following lemma we approximate $\F(\Om)$ functions using $L^{\infty}(\Om)$ functions similar result is obtained for $\mathcal{F}_{\frac{N}{p}}(\Om)$ in Proposition 3.2 of \cite{AMM}.
\begin{lem} \label{charF}
$g \in \F(\Om)$ if and only if for every $\epsilon >0,$ $\exists g_{\epsilon} \in L^{\infty}(\Om) $ such that $|Supp(g_{\epsilon})|< \infty$ and $\norm{g-g_{\epsilon}}< \epsilon.$
\end{lem}

\begin{proof}
Let $g \in \F(\Om)$ and $\epsilon >0$ be given. By definition of $\F(\Om)$, $ \exists g_{\epsilon} \in C_c^{\infty}(\Om)$ such that $\norm{g-g_{\epsilon}} < \epsilon .$ This $g_{\epsilon}$ fulfill our requirements. 
For the converse part, take a $g$ satisfying the hypothesis. Let $\epsilon >0 $ be arbitrary. Then $\exists g_{\epsilon} \in L^{\infty}(\Om)$ such that 
$|Supp(g_{\epsilon})|< \infty$ and $\norm{g-g_{\epsilon}}< \frac{\epsilon}{2}.$
Thus, $g_{\epsilon} \in L^{\frac{N}{p}}(\Om)$ and hence there exists $ \phi_{\epsilon} \in C_c^{\infty}(\Om)$ such that 
 $\norm{g_{\epsilon}-\phi_{\epsilon}}_{\frac{N}{p}} < \frac{\epsilon}{2C}$,
 where $C$ is the embedding constant for the embedding $ L^{\frac{N}{p}}(\Om)$ into $\H$. Now by triangle inequality, we obtain $\norm{g-\phi_{\epsilon}}< \epsilon$ as required.
\end{proof}

The next proposition gives an interesting property of capacity, which helps us to localize the norm on $\H$.
\begin{proposition} \label{propofcap1}
   There exists $C_1, C_2 >0$ such that for $F\cset \Om,$ 
   \begin{enumerate}[(i)]
    \item $ \cp(F \cap B_r(x), \Om \cap B_{2r}(x)) \leq C_1 \cp(F \cap B_{r}(x), \Om),\ \forall r>0.$
    \item $\cp(F \cap B_{2R}^c, \Om \cap \overline{B_R}^c) \leq C_2 \cp(F \cap B_{2R}^c, \Om), \ \forall R>0.$
   \end{enumerate}
\end{proposition}
\begin{proof}
$(i)$  Let $\Phi \in C_c^{\infty}(\R^N)$ be such that $0\leq \Phi \leq 1$, $\Phi = 1 $ on $\overline{B_1(0)}$ and $Supp(\Phi) \subseteq B_2 (0)$. Take $\Phi_r (z)= \Phi (\frac{z-x}{r}).$ Let  $\epsilon >0$ be given. 
 Then for $F\cset \Om,$ $\exists u \in \mathcal{N}(F \cap B_r(x))$ such that
$\int_{\Om}|\nabla u|^p < \cp(F \cap B_r(x), \Om) + \epsilon $. If we set $w_r(z)=\Phi_r(z) u(z) $, then it is easy to see that $w_r \in \D^{1,p}_0(\Om \cap B_{2r}(x))$ and $w_r \geq 1$ on $F \cap B_{r}(x)$. Further,
we have the following estimate:
  \begin{align*}
  \int_{\Om} |\nabla w_r|^p \ dx & \leq C \left[\int_{\Om} |\Phi_r|^p |\nabla u|^p \ dx + \int_{\Om} |u|^p |\nabla \Phi_r|^p\ dx \right] \\
  & \leq C \left[ \int_{\Om} |\nabla u|^p\ dx + \left(\int_{\Om} |u|^{p*}\ dx \right)^{p/p^*}  \left(\int_{\Om} |\nabla \Phi_r|^N\ dx \right)^{p/N} \right].    
                           \end{align*}
                           By noticing $\displaystyle \int_{\Om} |\nabla \Phi_r|^N \ dx \le \int_{\R^N} |\nabla \Phi|^N \ dx$ and then using  the Sobolev embedding, we obtain \[\int_{\Om} |\nabla w_r|^p \ dx \le C_1 \int_{\Om} |\nabla u|^p \ dx,\]                          
where $C_1$ is a constant  independent of $F,r$ and $\epsilon$. 
Therefore, \[\cp(F \cap B_r(x), \Om \cap B_{2r}(x)) \leq C_1 \cp(F \cap B_r(x), \Om) + C_1 \epsilon.\] Now as $\epsilon >0$ is arbitrary we obtain the desired result.\\
\noi $(ii)$ For $\Phi \in C_b^{\infty}(\R^N)$ with $0\leq \Phi \leq 1$, $\Phi = 0 $ on $\overline{B_1} (0)$ and $\Phi = 1 $ on $B_2 (0)^c$, we take
 $\Phi_R (z)= \Phi (\frac{z}{R}).$
The rest of the proof is similar to the proof of $(i)$.
\end{proof}

Now we consider the map $ |G|: \Dp \mapsto \R$ defined as $|G|(u)= \displaystyle \int_{\Om}|g||u|^p \ dx$ and state the following proposition.
\begin{proposition} \label{Gcpct}
 Let $g \in \H$. Then $G$ is compact  if and only if  $|G|$ is compact.
\end{proposition}
\begin{proof}
 Let $u_n \wra u$ in $\Dp$. Then $u_n \ra u$, $g|u_n|^p \ra g|u|^p$ and $|g||u_n|^p \ra |g||u|^p$ a.e in $\Om$. Further,
 \begin{eqnarray} \label{equality}
 |g |u_n|^p| = |g| |u_n|^p.
 \end{eqnarray}
Now as $g \in \H$, both of $g|u_n|^p$, $g|u|^p$ belong to $L^1(\Om)$. Since equality occurs in \eqref{equality}, a direct application of generalized dominated convergence theorem proves the required equivalence.
\end{proof}

\begin{lem} \label{cpct}
 Let $g\in \H$ and $G: \Dp \mapsto \R$ is compact. Then, 
 \begin{enumerate}
  \item[(i)] if $(A_n)$ is a sequence of bounded measurable subsets such that $\chi_{A_n}$ decreases to $0,$ then $\norm{g \chi_{A_n}} \ra 0$ as $n \ra \infty$.
  \item[(ii)] $\norm{g \chi_{B_n^c}} \ra 0$
  as $n \ra \infty$.
 \end{enumerate}
 \end{lem}
\begin{proof} $(i)$ Let $(A_n)$ be a sequence of bounded measurable subsets such that $\chi_{A_n}$ decreases to $0$. If $\norm{g \chi_{A_n}} \nrightarrow 0$, then
$\exists a>0$ such that $\norm{g \chi_{A_n}} >a, \forall n$ (by the monotonicity of the norm).
  Thus, $\exists F_n \cset \Om$ and $ u_n \in  \mathcal{N}(F_n)$ such that  
 \begin{equation} \label{convDp}
 \int_{\Om} |\nabla u_n|^p\ dx < \frac{1}{a} \int_{F_n \cap A_n} |g|\ dx \leq \frac{1}{a}\int_{\{|u_n| \geq 1 \}} |g| |u_n|^p\ dx.
 \end{equation}
 Since $A_n$'s are bounded and $\chi_{A_n}$ decreases to $0$, it follows that $|A_n| \ra 0$, as $n \ra \infty.$ Further, as $g \in L^1(A_1)$, we also have $\displaystyle \int_{F_n \cap A_n} |g|\ dx \ra 0.$ 
 Hence from the above inequalities, $u_n \ra 0$ in $\Dp$. For $0< \epsilon <1$, consider $ w_n^{\epsilon}=\displaystyle \frac{|u_n|^p}{(|u_n|+ \epsilon)^{p-1} \norm{u_n}_{\D}} $.
 One can check that for each $n$, $w_n^{\epsilon} \in \Dp$ and it is bounded uniformly (with respect to $n$) in $\Dp$. Thus up to a sub sequence, $w_n^{\epsilon}$ 
 converges weakly to $w$ in $\Dp$ as $n \ra \infty$. Now using the embedding of $\Dp$ into $L^{p^*}(\Om)$ we obtain that
 $\displaystyle \norm{w_n^{\epsilon}}_{\frac{p^*}{p}} \leq C \frac{\norm{u_n}_{\D}^{p-1}}{\epsilon^{(p-1)}}.$
 Thus $\norm{w_n^{\epsilon}}_{\frac{p^*}{p}} \ra 0$ as $n \ra \infty$ and hence $w=0$ i.e. $w_n^{\epsilon} \wra 0$ in $\Dp$
 as $n \ra \infty$. By the compactness of $|G|$ we infer $\displaystyle \lim_{n\ra \infty} \int_{\Om} |g||w_n^{\epsilon}|^p \ dx =0.$
 On the other hand, for each $n \in \N$ and $0 < \epsilon <1$,  
 \begin{eqnarray*}
  \int_{\Om} |g| |w_n^{\epsilon}|^p \ dx &=& \int_{\Om}  \frac{|g| |u_n|^{p^2}}{(|u_n|+ \epsilon)^{p^2-p} \norm{u_n}_{\D}^p}\ dx
   \geq  \int_{|u_n| \geq \epsilon} \frac{|g| |u_n|^{p^2}}{(2|u_n|)^{p^2-p} \norm{u_n}_{\D}^p} \ dx \\
  &=&  \frac{1}{2^{p^2-p}}\int_{\{|u_n| \geq \epsilon\}} \frac{|g| |u_n|^{p}}{\norm{u_n}_{\D}^p} \ dx> \frac{a}{2^{p^2-p}} 
 \end{eqnarray*}
which is a contradiction.
 
\noi $(ii)$ If  $\norm{g \chi_{B_n^c}} \nrightarrow 0$, as $n \ra \infty,$ then there exists $F_n \cset \Om$ such that
   \begin{eqnarray*}
a < \frac{\int_{F_n \cap B_n^c}  |g|\ dx}{\cp(F_n,\Om)} \leq \frac{\int_{F_n \cap B_n^c} |g|\ dx}{\cp(F_n \cap B_n^c,\Om)} 
\leq \frac{C\int_{F_n \cap B_n^c} |g|\ dx}{\cp(F_n \cap B_n^c,\Om \cap \overline{B}_{\frac{n}{2}}^c) }
\end{eqnarray*}
 for some $a>0$ and $C>0$. Last inequality follows from the part $(ii)$ of Proposition \ref{propofcap1}.
 Thus, for each $n$ there exists $z_{n} \in \D^{1,p}_0(\Om \cap \overline{B}_{\frac{n}{2}}^c)$ with $z_n \geq 1$ on $F_n \cap B_n^c$ such that 
\[ \int_{\Om} |\nabla z_n|^p \ dx < \frac{C}{a} \int_{F_n \cap B_n^c} |g| \ dx \leq \frac{C}{a} \int_{\Om} |g||z_n|^p \ dx.\]
 By taking $w_n=\displaystyle \frac{z_n}{\norm{z_n}_{\D}}$ and following a same argument as in $(i)$ we contradict the compactness of $|G|$ and hence, that of $G$. 
\end{proof}

Next   for $\phi \in C_c^{\infty}(\Om)$ we compute $\C_{\phi}$.
 \begin{proposition} \label{Cgzero}
  Let $\phi \in C_c^{\infty}(\Om)$. Then $\C_{\phi} \equiv 0.$
 \end{proposition}
 \begin{proof}
 First notice that for $\phi \in C_c^{\infty}(\Om),$
\begin{eqnarray*} 
  \norm{\phi \chi_{B_r(x)}} =  \sup_{F \cset \Om} \left[ \frac{\int_{F \cap B_r(x)}|\phi|\ dx}{\cp(F, \Om)}\right]  
 \leq  \sup_{F \cset \Om} \left[ \frac{ \sup (|\phi|) |(F\cap B_r)^{\star}|}{\cp((F \cap B_r)^{\star})}  \right]. \nonumber
 \end{eqnarray*}
 If $d$ is the radius of $(F \cap B_r)^{\star}$ then
 \[\frac{ |(F\cap B_r)^{\star}|}{\cp((F \cap B_r)^{\star})} = \frac{ \om_N d^N}{N \om_N (\frac{N-p}{p-1})^{p-1} d^{(N-p)}} = C(N,p) d^p \leq C(N,p) r^p .\]
 Thus, $\C_{\phi}(x) = \lim_{r \ra 0} \norm{\phi \chi_{B_r(x)}} =0 $. Also, one can easily see that $\C_{\phi}(\infty)=0 $ as $\phi$  has compact support .
 \end{proof}
 
 \begin{rmk} \label{boundedfunction} \rm
In fact, the same arguments as in the above proposition shows that $\C_g \equiv 0$ if $g \in L^{\infty}(\Om)$ and $g$ has compact support. 
 \end{rmk}

The next theorem  proves Theorem \ref{eqivthm}, Theorem \ref{eqivthm1} and Theorem \ref{eqivthm2} in one shot.
  \begin{thm}
 \label{allinone}
  Let $g \in \H$. Then the following statements are equivalent:
  \begin{enumerate}
   \item[(i)] $G: \Dp \mapsto \R $ is compact,
   \item[(ii)] $g$ has absolute continuous norm in $\H$,
   \item[(iii)] $g \in \F(\Om)$,
   \item[(iv)] $\C^*_g =0= \C_g(\infty)$.
  \end{enumerate}
  \end{thm}
\begin{proof}
 $(i) \implies (ii):$ Let $G$ be compact. Take a sequence of measurable subsets $(A_n)$ 
 of $\Om$ such that $\chi_{A_n}$ decreases to $0$ a.e. in $\Om$. Part $(ii)$ of Lemma \ref{cpct} gives $\norm{g \chi_{B_n^c}} \ra 0$, as $n \ra \infty$. Choose $\epsilon >0$ arbitrarily. There exists $N_0 \in \N,$ 
such that $\norm{g \chi_{B_n^c}} \leq \frac{\epsilon}{2}, \forall n \geq N_0.$ Now $A_n= (A_n \cap B_{N_0}) \cup (A_n \cap B_{N_0}^c)$,  for each $n$. Thus, 
\[\norm{g \chi_{A_n}} \leq \norm{g \chi_{A_n \cap B_{N_0}}} + \norm{g \chi_{A_n \cap B_{N_0}^c}} \leq \norm{g \chi_{A_n \cap B_{N_0}}} + \frac{\epsilon}{2}.\]
 By part $(i)$ of Lemma \ref{cpct}, there exists $N_1(\geq N_0) \in \N$ such that $\norm{g \chi_{A_n \cap B_{N_0}}} \leq \frac{\epsilon}{2}, \ \forall n \geq N_1$ and hence $\norm{g \chi_{A_n}} \leq \epsilon$ for all $n \geq N_1$. 
 Therefore, $g$ has absolutely continuous norm.

\noi $(ii) \implies (iii):$ Let $g$ has absolute continuous norm in $\H$. Then, $\norm{g \chi_{B_m^c}} $ converge to $0$ as $m \ra \infty$. Let $\epsilon >0$ be arbitrary. We choose $m_{\var} \in \N$ such that $\norm{g \chi_{B_m^c}} < \epsilon$, $\forall m \geq m_{\var}$. Now for any $n \in \N$,
\[g = g \chi_{\{|g| \leq n\} \cap B_{m_{\var}}} + g \chi_{\{|g| >n\} \cap B_{m_{\var}}} + g \chi_{B_{m_{\var}}^c} := g_n + h_n.\]
where $g_n=g \chi_{\{|g| \leq n\} \cap B_{m_{\var}}}$ and $h_n=g \chi_{\{|g| >n\} \cap B_{m_{\var}}} + g \chi_{B_{m_{\var}}^c}.$ Clearly, $g_n \in L^{\infty}(\Om)$ and $|Supp(g_n)| < \infty $. 
Furthermore,
\[ \norm{h_n} \leq \norm{g \chi_{\{|g| >n\} \cap B_{m_{\var}}}} + \norm{g \chi_{B_{m_{\var}}^c}} <  \norm{g \chi_{\{|g| >n\} \cap B_{m_{\var}}}} + \epsilon \,. \]
 Now, $g \in L^1_{loc}(\Om)$ ensures that $\chi_{\{|g| >n\} \cap B_{m_{\var}}} \ra 0$ as $n\ra \infty$. As $g$ has absolutely continuous norm, $\norm{g \chi_{\{|g| >n\} \cap B_{{m_{\var}}}}} < \epsilon$ for large $n$. Therefore, $\norm{h_n}< 2\epsilon$ for large $n$.
 Hence, Lemma \ref{charF} concludes that $g \in \F(\Om)$.
 
 \noi $(iii) \implies (iv):$ Let $g \in \F(\Om)$ and $\epsilon >0$ be arbitrary. Then there exists $g_\var \in C_c^{\infty}(\Om)$ such that
 $\norm{g-g_\var} < \epsilon$. Thus Proposition \ref{Cgzero} infers that $\C_{g_\var}$ vanishes.
Now as $g = g_\var + (g-g_\var)$, it follows that $\C_g(x) \leq \C_{g_\var}(x) + \C_{g - g_\var}(x) \leq \norm{g - g_\var} < \epsilon$ and hence $\C^*_g=0$. By a similar argument one can show $\C_g(\infty)=0.$

\noi $(iv) \implies (i):$ Assume that $\C^*_g =0= \C_g(\infty)$. Let $(u_n)$ be a bounded sequence in $ \Dp$. Then by Lemma \ref{mlc2}, up to a sub-sequence we have,
 \begin{eqnarray*}
 \nu_{\infty} &\leq& C_H\ \C_g(\infty)  \Gamma_{\infty} \label{1},\\
 \norm{\nu} &\leq& C_H \C^{*}_g \norm{\Gamma} \label{2}, \\
 \lim_{n \ra \infty} \int_{\Om} |g||u_n|^p \ dx &=& \int_{\Om} |g||u|^p \ dx + \norm{\nu} + \nu_{\infty} \label{3}.
 \end{eqnarray*}
 As $\C^*_g=0= \C_g(\infty)$ we immediately conclude that
 $\displaystyle \lim_{n \ra \infty} \int_{\Om} |g||u_n|^p \ dx = \int_{\Om} |g||u|^p \ dx $
 and hence $G: \Dp \mapsto \R$ is compact (Proposition \ref{Gcpct}).
 \end{proof}
\begin{rmk}[Rellich compactness theorem]  \rm 
Let $\Om$ be a bounded domain in $\R^N$ and $g\equiv 1$ on $\Om$. Then, by Remark \ref{boundedfunction}, $\C_g \equiv 0$ and hence, by the above equivalence $G$ is compact on $\Dp$ i.e., $\Dp$ is compactly embedded into $L^p(\Om)$.
\end{rmk}

\begin{rmk} \rm \label{nexistrmk}
 Let $N>p$ and $g(x)= \frac{1}{|x|^p}$ in $\R^N$.
Then for any $r>0$, using Proposition \ref{propcap} we get
 \[\displaystyle \frac{\int_{B_r(0)} \frac{dx}{|x|^p}}{\cp(B_r(0))} = \frac{(p-1)^{p-1}}{(N-p)^p}. \] 
 Thus $\C_g(0) = \frac{(p-1)^{p-1}}{(N-p)^p}$ and hence $\displaystyle g \notin \F(\R^N).$
\end{rmk}
\begin{rmk} \rm \label{contabscont}
   Let $X = (X(\Om),\norm{.}_X )$ be a Banach function space and $f \in X$. Then $f$ is said to have continuous norm in $X$, if for each $x \in \Om$,  $\norm{f \chi_{B_r(x)}}$ converges to $ 0  $, as $r \ra 0$. 
   Observe that by  Theorem \ref{allinone}, the set of all functions having continuous norm and the set of all function having absolute continuous norm are one and the same on $\H$. However, in \cite{Lang},
  authors constructed a Banach function space where these two sets are different.
 \end{rmk}
 
 Now, we recall Maz'ya's concentration function $\Pi_g$ ,(see Section 2.4.2, page 130 of \cite{Mazya}).
 For $F \cset \Om$ with $|F| \neq 0$, let $\Pi (F,g,\Om):= \frac{\int_{F}|g|\ dx}{\cp(F,\Om)}$. Then
\begin{align*}
   \Pi_g(x)  & =  \lim_{r \ra 0}  \sup \{ \Pi(g,F,\Om) :  F \cset \Om \cap B_r(x) \},  \\
  \Pi_g(\infty)  &=  \lim_{r \ra \infty} \ \sup \left\{ \ \Pi(g,F,\Om) : F\cset \Om \cap B_r(0)^c \right\}.
 \end{align*}
 Next proposition shows that  $\C_g$ coincides with $\Pi_g$. As  $\C_g$ measures the concentration using the norm of $\H$, we prefer $\C_g$ over $\Pi_g.$
\begin{proposition} \label{unilemma}
 Let $g \in \H$. Then $\C_g(x) = \Pi_g(x)$ for all $x \in \overline{\Om}$ and $\C_g(\infty)= \Pi_g(\infty)$.
 \end{proposition}
 \begin{proof}
First notice that  $\Pi_g(x) \leq \C_g(x)$, for any $x \in \overline{\Om}$ and $\Pi_g(\infty) \leq \C_g(\infty)$.
On other hand for $V \cset \Om$,
\begin{eqnarray*}
 \displaystyle \frac{\int_{V} |g| \chi_{B_r(x)}\ dx}{\cp(V,\Om)} \leq \frac{\int_{V \cap B_r(x)} |g| \ dx}{\cp(V \cap B_r(x),\Om)} \leq \sup_{F \cset \Om \cap B_{2r}(x)} \left[ \frac{\int_{F}|g|\ dx}{\cp(F, \Om)} \right]=\Pi^{2r}_g(x).
\end{eqnarray*}
The last inequality follows as $V \cap B_r(x)$ is relatively compact in $ \Om \cap B_{2r}(x)$. Taking the supremum over all $V \cset \Om$ and letting $r \ra 0$ we obtain
$\C_g(x) \leq \Pi_g(x)$. By a similar argument we also get $\C_g(\infty) \leq \Pi_g(\infty)$ as required.  
\end{proof}

\section{A concentration compactness criteria} \label{Existence}
  Recall that, for $g \in \H$,   the best constant $B_g$ in \eqref{HS} is given by
 \begin{eqnarray*}
 \frac{1}{B_g} = \inf_{u\in  G^{-1}\{1\}}  \int_{\Om} |\nabla u|^p \ dx.
 \end{eqnarray*} 

	\noi{\bf Proof of Theorem \ref{exismin}.}
Let $(u_n) \in G^{-1}\{1\}$ be a sequence that minimizes $ \displaystyle \int_{\Om} |\nabla u|^p\ dx$ over $G^{-1} \{1\}.$ Then up to a sub-sequence
we can assume that  $ u_n \wra u $ in $ \Dp$ and $u_n \ra u$ a.e. in $\Om$. Further, $ |\nabla u_n - \nabla u|^p \wrastar \Gamma $, $|\nabla u_n| \wrastar \tilde{\Gamma}$, $g|u_n - u|^p \wrastar \nu$  in $\mathbb{M}(\R^N)$. 
Since $u_n\in  G^{-1} \{1\}$, using  Lemma \ref{mlc2} we have \[1= \int_{\Om} g |u|^p\ dx + \norm{\nu} + \nu_{\infty} .\] 
Suppose  $\norm{\nu}$ or $\nu_{\infty}$ is nonzero. Then $\C_g^*$ or $\C_g(\infty) \neq 0$ respectively. Now using Hardy-Sobolev inequality and Lemma \ref{mlc2}, we obtain the following estimate:
\begin{align*}
1=B_g \times \overline{\lim}_{n \ra \infty} \int_{\Om} |\nabla u_n|^p\ dx &\geq  B_g \left[ \int_{\Om} |\nabla u|^p \ dx+  \frac{\norm{\nu}}{C_H \C_g^*} + \Gamma_{\infty} \right] \\
& \geq B_g \left[ \frac{1}{B_g} \int_{\Om} g|u|^p \ dx +  \frac{\norm{\nu}}{C_H \C_g^*} + \frac{\nu_{\infty}}{C_H \C_g(\infty)} \right] \\
& > \frac{B_g}{B_g} \left[\int_{\Om} g|u|^p\ dx +  \norm{\nu} + \nu_{\infty}\right],
\end{align*}
a contradiction. Thus $\norm{\nu}=0=\nu_{\infty}.$
Therefore, $\displaystyle \int_{\Om} g|u|^p\ dx=1$ and consequently, $B_g$ is attained at $u$. 

\begin{rmk} \rm
 For $g(x)= \frac{1}{|x|^p}$ in $\R^N$, it is well known that $B_g$ is not attained in $\Dp$. Further, $ \C_g(0)= \frac{(p-1)^{p-1}}{(N-p)^p}$ and hence $C_H \C^{*}_g = B_g.$
\end{rmk}

\begin{rmk} \rm \label{weakrmk}
Recall the definition of $S_g(x)$. In \cite{Smets}, author also considered the following quantities : 
\begin{align*}
  S_{g}(x) & :=  \lim_{r\ra 0} \inf \left \{\int_{\Om} |\nabla u|^p \ dx : u\in \D^{1,p}_0 (\Om \cap B_r(x)), \ \int_{\Om} g|u|^p \ dx =1 \right \} \,, \\
 S_g^* & :=  \sup_{x \in \overline{\Om}} S_g(x),
 \\
 S_{g}(\infty) & :=  \lim_{R\ra \infty} \inf \left \{\int_{\Om} |\nabla u|^p \ dx : u\in \D^{1,p}_0 (\Om \cap B_R^c), \ \int_{\Om} g|u|^p \ dx =1 \right \}, \\ 
  S_{g} & :=   \inf \left \{\int_{\Om} |\nabla u|^p \ dx: u\in \D^{1,p}_0 (\Om), \ \int_{\Om} g|u|^p \ dx =1 \right \}.
 \end{align*}
 Since $S_g(.)$ captures the best constant in the Hardy inequality locally at the points of $\Om$ and at the infinity, by  \eqref{estimate}, we have 
\begin{eqnarray}\label{S_gC_g}
 \norm{g} \leq \frac{1}{S_g} \leq C_H \norm{g}, \quad \C^*_g \leq \frac{1}{S_{g}^*} \leq C_H \C^*_g, \quad \C_g(\infty) \leq \frac{1}{S_g(\infty)} \leq C_H \C_g(\infty).
\end{eqnarray}
 Therefore, if $C_H \C^*_g < \norm{g} $ and $C_H \C_g(\infty) < \norm{g}$ then  $S_g < S_{g}^*$ and $S_g < S_g(\infty) $. Thus, if in addition $\overline{\sum_g}$ is countable, then Theorem \ref{exismin} also follow from Theorem 3.1 of \cite{Smets}. Therefore, our sufficient condition is slightly weaker than that of \cite{Smets}. This is mainly because of the gap in the Hardy inequality given in \ref{Mazya's condition} (see \eqref{estimate}). However, on the other hand, our sufficient condition assumes $|\overline{\sum_g}|=0$ instead of its countability. 
\end{rmk}

\begin{rmk} \rm
In \cite{Smets}, Smets proved the Mazya's compactness criteria by showing  that $G$ is compact if and only if $S_g^* = S_g(\infty)= \infty$. Observe that, one can easily derive this result by using \eqref{S_gC_g} together with Theorem \ref{eqivthm2}.
\end{rmk}

\noi{\bf Proof of Theorem \ref{thmperturb}.} Let $h \in \H$ be non-negative and $|\overline{\sum_h}|=0$. Take a non-zero, non-negative $\phi \in \F(\Om)$ and $\epsilon_0 = \frac{(2C_H-1) \norm{h}}{\norm{\phi}},$ then for $ \epsilon > \epsilon_0 $, let $g = h+ \epsilon \phi.$
Clearly, $|\overline{\sum_g}|=0$ and 
\begin{eqnarray*}
  C_H \C^{*}_{g} =C_H \C^{*}_{h+ \epsilon \phi} = C_H \C^{*}_h \leq C_H \norm{h} <  \frac{\norm{h} + \epsilon \norm{\phi}}{2} \leq \norm{g} \leq B_g. 
 \end{eqnarray*}
 Similarly, we can show $C_H \C_{g}(\infty) < \norm{g } \leq B_g.$ Therefore, by Theorem \ref{exismin}, $B_g$ is attained.

\begin{rmk} \label{example} \rm 
$(i)$. For $2 \leq k < N$ and for $z\in \R^N,$ we write $z=(x,y)\in \R^{k} \times \R^{N-k}$. Now consider $g(z)=\frac{1}{|x|^p}$ in $\R^{k} \times \R^{N-k}.$ By Theorem 2.1 of \cite{Tarantello},  $g \in \mathcal{H}(\R^N)$ if $p<k$. Next we show that $\sum_g = \{0\} \times \R^{N-k}$. For  any $(0,y) \in \R^{k} \times \R^{N-k}$ and $r>0$, using the translation invariance of both the integral and the $\cp$, we have
$$\frac{\int_{B_r(0,y)}g(z) \ dz}{\cp(B_r(0,y))} = \frac{\int_{B_r(0,0)} \frac{1}{|x|^p} \ dz}{\cp(B_r(0,0))}\ge \frac{\int_{B_r(0,0)} \frac{1}{|z|^p} \ dz}{\cp(B_r(0,0))}.$$
Now by taking $r \ra 0$ we have $\C_g(0,y) \geq \C_{\frac{1}{|z|^p}}((0,0))>0$ and hence $\sum_g \supseteq \{0\} \times \R^{N-k}$. Next for $z_0=(x_0,y_0) \notin \{0\} \times \R^{N-k}$,  let $0<r<|x_0.|$ Then by Proposition \ref{propcap} we obtain
$$\frac{\int_{B_r(z_0)} \frac{1}{|x|^p} \ dz}{\cp(B_r(z_0))} \leq \frac{ \frac{1}{(|x_0|-r)^p} \int_{B_r(z_0)} \ dz}{\cp(B_r(z_0))}= \left(\frac{p-1}{N-p}\right)^{p-1} \left(\frac{r^p}{N(|x_0|-r)^p } \right).$$
Now by taking $r \ra 0$, we obtain $\C_g(z_0)=0$.
Hence, $\sum_g=\{0\} \times \R^{N-k}$. 

\noi $(ii)$. Let $2\leq k<N$, $p<k$. We consider $g(z)=\frac{1}{|x|^p},$ for $z=(x,y) \in \R^k \times \R^{N-k}$. In Example \ref{example}, we have seen that
$g \in \H$ with $\sum_g$ is uncountable and $|\overline{\sum_g}|=0$. Now choose any $\phi \in \F(\Om)$ such that $\norm{\phi}=2(2C_H-1) \norm{g}$ and consider $\tilde{g}:=g + \phi$. Then  $\epsilon_0=\frac{1}{2}$ and hence that $B_{\tilde{g}}$ is attained (by Theorem \ref{thmperturb}). Further, in Example \ref{example}, we have seen that $\sum_{\tilde{g}}$ is also uncountable and $|\overline{\sum_{\tilde{g}}}|=0$. Thus, $\tilde{g}$ lies outside the class of functions considered in \cite{Smets,Tertikas}.
\end{rmk}

\begin{corollary} \label{distrmk}
Let $g \in \H$ with $g \geq 0$ and $|\overline{\sum_g}|=0$.
If $C_H  dist(g, \F(\Om)) < \norm{g} ,$ then $B_g$ is attained in $\Dp$.
\end{corollary}
\begin{proof}
 For $ g, h \in L^1_{loc}(\Om)$ and $F \cset \Om $, 
 \[ \frac{\int_F |g| \chi_{B_r(x)}\ dx}{\cp(F, \Om)} \leq \frac{\int_F |g-h| \chi_{B_r(x)}\ dx}{\cp(F, \Om)} + \frac{\int_F |h| \chi_{B_r(x)}\ dx}{\cp(F, \Om)}. \]
 By taking the supremum over all such $F$ and $r$ tends to $0$ respectively, we obtain $\C_g(x) \leq \C_{g-h}(x) + \C_{h}(x) $ and hence
 \begin{eqnarray} \label{claim}
 \C^{*}_g \leq \C^{*}_{g- h} + \C^{*}_{h}. 
 \end{eqnarray}
Now as $C_H  dist( g, \F(\Om)) < \norm{g} $,
 $\exists \phi \in \F(\Om)$ such that $C_H \norm{g - \phi} < \norm{g}.$ Thus by \eqref{claim},
 $C_H \C^{*}_{g} \leq  C_H \C^{*}_{g - \phi} \leq C_H \norm{g - \phi} < \norm{g} \leq B_g$ and similarly $C_H \C_{g}(\infty) <  B_g .$  Now the result follows from Theorem \ref{exismin}.
\end{proof}

Next proposition also gives us another way to produce the Hardy potential for which $B_g$ is attained in $\Dp$ without $G$ being compact.
 \begin{proposition}\label{posipartrmk}
  Let $g \in L^1_{loc}(\Om)$ be such that $g^+ \in \F(\Om)$. Then the best constant $B_g$ is attained.
 \end{proposition}
\begin{proof}
 Let $(u_n)$ be a sequence that minimizes $\displaystyle \int_{\Om} |\nabla u|^p \ dx$ over $G^{-1} \{1\}$. Then $(u_n)$ is bounded in $\Dp$ and hence up to a subsequence $u_n \wra u$ in $\Dp$ and $u_n \ra u$ a.e. in $\Om$. Since $g^+ \in \F(\Om)$, $\displaystyle \lim_{n \ra \infty} \int_{\Om} g^+ |u_n|^p \ dx = \int_{\Om} g^+ |u|^p \ dx$ (by Theorem \ref{eqivthm}). Further, $\displaystyle \int_{\Om} g^- |u_n|^p \ dx =  \int_{\Om} g^+ |u_n|^p \ dx -1. $ 
 Now Fatous lemma gives $\displaystyle  \int_{\Om} g^- |u|^p\ dx \leq  \int_{\Om} g^+ |u|^p\ dx -1 $. Thus $1 \leq \int_{\Om} g |u|^p\ dx$ and hence $\tilde{u} := \frac{u}{[\int_{\Om} g |u|^p\ dx]^{1/p}}$ is a required element in $\Dp$ for which the best constant
 is attained.
\end{proof} 
\begin{rmk} \rm
 The above proposition gives an alternate way to produce examples of weight function $g$ for which  the best constant $B_g$ is attained without $G$ being compact. For example, take $g$ in $L^1_{loc}(\Om)$ with $g^+\in \F(\Om)$ and $g^-\notin \F(\Om).$
\end{rmk}

\section{Acknowledgement} T. V. Anoop  would like to  thank the Department of Science \& Technology, India for the research grant DST/INSPIRE/04/2014/001865.
\bibliographystyle{abbrv}
\bibliography{ref}

 \vspace{0.8cm}
\noi {\bf  T. V.  Anoop } \\  Department of Mathematics,\\   Indian Institute of Technology Madras, \\ Chennai, 600036, India. \\ 
{\it Email}:{ anoop@iitm.ac.in} \\

		\noi {\bf Ujjal Das } \\  The Institute of Mathematical Sciences, HBNI \\ Chennai, 600113, India. \\ 
{\it Email}:{ujjaldas@imsc.res.in, ujjal.rupam.das@gmail.com}

\end{document}